\documentclass[leqno,12pt]{amsart}
\usepackage[letterpaper,margin=1in]{geometry}
\usepackage{amsmath,amssymb,amsthm,mathtools}
\usepackage{microtype}
\usepackage[hidelinks]{hyperref}
\usepackage{xcolor}
\numberwithin{equation}{section}

\newtheorem{theorem}{Theorem}[section]
\newtheorem{corollary}{Corollary}[section]
\newtheorem{lemma}[theorem]{Lemma}
\newtheorem{remark}[theorem]{Remark}
\newcommand{\R}{\mathbb{R}}

\newcommand{\weakL}[1]{L^{#1,\infty}(\R)}

\makeatletter
\@namedef{subjclassname@2020}{\textup{2020} Mathematics Subject Classification}
\makeatother

\title{Optimal Decay Exponents for Fractional Burgers Rarefaction Waves}
\author[T. Phan]{
Tuoc Phan}
\address[T. Phan]{Department of Mathematics, University of Tennessee, Knoxville, TN, USA.}\email{tphan2@utk.edu}

\author[A. Stefanov]{Atanas Stefanov}
\address[A. Stefanov]{Department of Mathematics, University of Alabama at Birmingham, Birmingham, AL, USA.}\email{stefanov@uab.edu}%


\thanks{T. Phan is partially supported by Simons Foundation  grant 769369, A. Stefanov is partially supported through an NSF-DMS grant  2605398.}

\begin{document}
\maketitle
\vspace{-0.8cm}
\begin{abstract}
We establish decay estimates toward rarefaction waves for the one-dimensional fractional Burgers equation when the order of dissipation lies between one and two. The initial data are bounded, and their perturbation from the Riemann step belongs to the critical weak-Lorentz space determined by scaling; the perturbation need not be small or integrable. We obtain optimal algebraic decay exponents, improving the known decay rates, and extending convergence to every Lebesgue space above the critical exponent, including the supremum norm. We prove optimality by constructing a single solution with Riemann step initial data whose lower decay estimates have exactly the same algebraic exponents as our upper estimates. These lower bounds establish that neither the algebraic decay exponents nor the range of convergence can be improved under the stated assumptions. The upper estimates contain a logarithmic factor whose necessity remains open. The proof combines convex truncation estimates, fractional interpolation, and a one-sided slope estimate for the supremum norm.
\end{abstract}

\section{Introduction and main results}\label{sec:introduction}

\subsection{Problem setting and main result}\label{sec:setting}

We study convergence toward rarefaction waves for the one-dimensional fractional Burgers equation
\begin{equation}\label{eq:fractional-burgers}
\left\{
\begin{array}{cccl}
u_t+\Lambda^\alpha u+uu_x & = & 0, &
\quad (t,x)\in(0,\infty)\times\R,\\
u(0, x) & = & u_0(x), & \quad x \in \R,
\end{array} \right.
\end{equation}
where $\Lambda^\alpha=(-\partial_x^2)^{\alpha/2}$ is the Fourier multiplier with symbol $|\xi|^\alpha$ for $\alpha \in (1,2)$. Our aim is to determine the algebraic decay rate and the range of Lebesgue spaces in which the solution $u$ converges to the rarefaction wave for bounded initial data $u_0$, allowing nonintegrable initial perturbations.

\smallskip
To state our result, let us recall that for given constant states $u_-<u_+$, the associated rarefaction wave is
\begin{equation}\label{eq:rarefaction}
w^R(t,x)=
\begin{cases}
u_-,&x\le u_-t,\\
x/t,&u_-t<x<u_+t,\\
u_+,&x\ge u_+t,
\end{cases}
\qquad t>0.
\end{equation}
This is the entropy solution of the inviscid Burgers equation with Riemann step initial data, which we denote by $w^R(0, \cdot)$ as in \eqref{w-R0} below. 

\smallskip
We set $q_c=1/(\alpha-1)>1$ and we measure the perturbation of the initial data in the critical weak-Lorentz space $L^{q_c,\infty}(\R)$. Here, for each $q\in (1,\infty)$, the weak-Lorentz quasi-norm is defined by
\begin{equation} \label{Lorentz-norm-def}
\|f\|_{L^{q,\infty}(\R)}
=\sup_{\lambda>0}\lambda\,
\bigl|\{x\in\R:|f(x)|>\lambda\}\bigr|^{1/q}.
\end{equation}
As we point out in Remark \ref{rem:scaling} below,
 the equation \eqref{eq:fractional-burgers} is invariant under the scaling $u(t,x) \mapsto \lambda^{\alpha-1}u(\lambda^\alpha t,\lambda x)$ for $\lambda>0$, and the spaces $L^{q_c}(\R)$ and $L^{q_c,\infty}(\R)$ are invariant under the scaling. This is an important starting point for this work and the decay estimates are obtained in $L^p(\R)$ for $p \in (q_c, \infty]$.
 
\smallskip
In Theorem~\ref{thm:optimality}, we improve the decay rates in \cite{KMX2008} and extend convergence to $L^p(\R)$ for every $p\in(q_c,\infty]$.  Applications of this result will be given in our forthcoming work \cite{PS}, which also gives improved decay rates for the parabolic-elliptic system arising in the study of hydrodynamics of radiating gas, subject to (Gaussian and fractional) viscosity and to thermal conductivity. We give two direct applications with improved decay rates in Section~\ref{sec:applications}. 

\smallskip
Throughout the paper, the numbers $C$ and $C_p$ denote positive constants that may change from line to line and are independent of time. They may depend on $\alpha$, the end states $u_{\pm}$, and the norms of the initial data specified in the relevant statement. We also note that the existence of unique bounded smooth solutions $u$ to \eqref{eq:fractional-burgers} is known and it will be recalled in Theorem~\ref{thm:existence}.
\begin{theorem}\label{thm:optimality}
Let $\alpha \in (1,2)$ and $u_-<u_+$. Let $w^R$ be the rarefaction wave \eqref{eq:rarefaction}, whose initial datum is
\begin{equation} \label{w-R0}
w^R(0,x)=
\begin{cases}
u_-,&x<0,\\
u_+,&x\ge0.
\end{cases}
\end{equation}
For each $p \in (q_c, \infty]$, where $q_c=1/(\alpha-1)$, suppose that $u_0\in L^\infty(\R)$ satisfies
\[
u_0-w^R(0, \cdot)\in L^p(\R)\cap L^{q_c,\infty}(\R).
\]
Then the unique bounded solution $u$ of \eqref{eq:fractional-burgers} satisfies, with $1/\infty=0$,
\begin{equation} \label{main-thrm-decay-est}
\|u(t, \cdot)-w^R(t, \cdot)\|_{L^p(\R)}
\le C_p(1+t)^{-\frac{\alpha-1-1/p}{\alpha}}
\bigl[1+\log(2+t)\bigr],\qquad t\ge0,
\end{equation}
where
\[
C_p = C\Bigl(
1+\|u_0-w^R(0, \cdot)\|_{L^p(\R)}
+\|u_0-w^R(0, \cdot)\|_{L^{q_c,\infty}(\R)}\Bigr),
\]
with $C= C(\alpha,p,u_-,u_+)>0$ depending only on the indicated parameters. If the estimate \eqref{main-thrm-decay-est} is restricted to $t\ge1$, the term $\|u_0-w^R(0, \cdot)\|_{L^p(\R)}$ in $C_p$ can be omitted for every $p \in (q_c, \infty]$.
\end{theorem}
\noindent
We note that the exponent $\frac{\alpha-1-1/p}{\alpha}$ in the decay estimate \eqref{main-thrm-decay-est} and the range $(q_c, +\infty]$ of $p$ coincide those in the known results when $\alpha =2$ established in \cite{HN1991} and \cite[Propositions~3.2--3.3, p.~224]{KS}, see Theorem \ref{thm:viscous-profile} below. The following remark establishes optimality of the algebraic exponent and the range of $p$ in the estimate \eqref{main-thrm-decay-est}.
\begin{remark} \label{optimal-remark}
The algebraic exponent in the decay estimate \eqref{main-thrm-decay-est} is optimal. More precisely, for $u_-=-1$, $u_+=0$, and the exact step initial datum
\[
u_0(x)=w^R(0,x)=
\begin{cases}
-1,&x<0,\\
0,&x\ge0,
\end{cases}
\]
the unique bounded solution satisfies, for every $p \in [1,\infty]$,
\begin{equation}\label{eq:optimality-lower}
\|u(t, \cdot)-w^R(t, \cdot)\|_{L^p(\R)}
\ge c\,t^{-\frac{\alpha-1-1/p}{\alpha}},\qquad t\ge1,
\end{equation}
where $c>0$ and $\alpha \in (1, 2]$. Thus the algebraic exponent cannot be increased, and convergence fails in general for $p \in [1,q_c]$. The necessity of the logarithmic factor in the upper bound remains open.  The estimate \eqref{eq:optimality-lower} when $\alpha =2$ also proves the optimality of the known decay rates established in \cite{HN1991} and \cite[Propositions~3.2--3.3, p.~224]{KS} that are recorded in Theorem \ref{thm:viscous-profile} below.
\end{remark}
\smallskip
We note that the initial perturbation $u_0-w^R(0, \cdot)$ need not be small, monotone, or in $L^1(\R)$. For example, the hypotheses permit tails of order $(1+|x|)^{-(\alpha-1)}$, which belong to $L^{q_c,\infty}(\R)$ and to every $L^p(\R)$ with $p \in (q_c,\infty]$, but not to $L^{q_c}(\R)$ or $L^1(\R)$. The critical exponent $q_c$ and the decay rate are consistent with the scaling of the equation, as we explain in the following remark.
\begin{remark}\label{rem:scaling}
The fractional Burgers equation \eqref{eq:fractional-burgers} is invariant under
\[
u_\lambda(t,x)=\lambda^{\alpha-1}u(\lambda^\alpha t,\lambda x),
\qquad \lambda>0,
\]
since all three terms scale by $\lambda^{2\alpha-1}$. Note also that for any $f\in L^p(\R)$,
\[
\|\lambda^{\alpha-1}f(\lambda\,\cdot)\|_{L^p(\R)}
=\lambda^{\alpha-1-1/p}\|f\|_{L^p(\R)}, \quad p \in [1, \infty]
\]
and the same scaling law holds for the weak-Lorentz quasi-norm. Thus $L^{q_c}(\R)$ and $L^{q_c,\infty}(\R)$ are invariant under this spatial scaling. The corresponding decay exponent is $(\alpha-1-1/p)/\alpha$.
\end{remark}

\subsection{Relation to earlier work}

For the classical viscous Burgers equation, Hattori and Nishihara~\cite{HN1991} studied convergence toward rarefaction waves and the different behavior near and away from their corners. The viscous-profile bounds used in this paper are recorded in \cite[Lemma~2.4]{KMX2008} and \cite[Propositions~3.2--3.3, p.~224]{KS}. These estimates allow us to compare the fractional Burgers solution with a smooth monotone profile before returning to the inviscid rarefaction wave.

\smallskip
The closest earlier result for \eqref{eq:fractional-burgers} is due to Karch, Miao, and Xu~\cite[Theorem~1.1]{KMX2008}. For $\alpha\in(1,2)$, they considered bounded-variation initial data satisfying $u_0-w^R(0,\cdot)\in L^1(\R)$ and proved, for $p\in((3-\alpha)/(\alpha-1),\infty]$,
\begin{equation}\label{eq:earlier-rarefaction-decay}
\|u(t,\cdot)-w^R(t,\cdot)\|_{L^p(\R)}
\le C_p t^{-\frac12(\alpha-1-(3-\alpha)/p)}\log(2+t),\qquad t>0.
\end{equation}
\smallskip
An important related estimate concerns the difference of two solutions of the fractional equation when $\alpha \in (0, 1]$. If $u_0,\widetilde u_0\in L^\infty(\R)$ are function with bounded variation, $\widetilde u_0$ is nondecreasing, and $u_0-\widetilde u_0\in L^1(\R)$, then, for $p\in[1,\infty]$ and $t>0$,
\begin{equation}\label{eq:earlier-solution-comparison}
\|u(t,\cdot)-\widetilde u(t,\cdot)\|_{L^p(\R)}
\le C_{\alpha,p}t^{-\frac1\alpha(1-1/p)}
\|u_0-\widetilde u_0\|_{L^1(\R)}.
\end{equation}
See \cite[Theorem~1.1]{AIK2010}. In \eqref{eq:earlier-solution-comparison}, the reference function $\widetilde u$ solves the fractional equation \eqref{eq:fractional-burgers} when $\alpha \in (0,1]$ with corresponding initial data $\widetilde u_0$, whereas in \eqref{eq:optimality-lower} $w^R$ solves the inviscid equation. Thus one must also estimate $\widetilde u-w^R$ to obtain rarefaction-wave convergence. This distinction is relevant to our lower bound in \eqref{eq:optimality-lower}: even exact Riemann step data generate a nonzero error relative to $w^R$.

\smallskip
For data satisfying both sets of hypotheses, Theorem~\ref{thm:optimality} improves the algebraic exponent in \eqref{eq:earlier-rarefaction-decay}, since
\[
\frac{\alpha-1-1/p}{\alpha}
-\frac12\Big(\alpha-1-\frac{3-\alpha}{p}\Big)
=\frac{(2-\alpha)(\alpha-1)(1+1/p)}{2\alpha}>0.
\]
We also extend convergence to every $p\in(q_c,\infty]$ under a weak critical integrability condition on the initial perturbation, allowing nonintegrable tails. Remark~\ref{optimal-remark} establishes the optimality of this algebraic exponent and of the range of convergence. The necessity of the logarithmic factor remains open.

\smallskip
The order of dissipation determines the asymptotic profile. For bounded data with an integrable deviation from the Riemann step, Alibaud, Imbert, and Karch~\cite[Theorems~1.3, 1.5, and~1.6]{AIK2010} identified a linear fractional-diffusion profile when $\alpha\in(0,1)$ and a nonlinear self-similar profile when $\alpha=1$. Karch, Pude\l{}ko, and Xu~\cite{KPX2015} studied two-dimensional fractional convection-diffusion equations with step-like data. They obtained rarefaction, diffusion, or nonlinear self-similar asymptotics depending on the dissipation order and the direction of convection. These results place the present one-dimensional rarefaction problem within the broader classification of asymptotic profiles.

When the end states satisfy $u_->u_+$, the large time profile is a viscous shock, rather than a rarefaction wave. For scalar conservation laws with Riesz-Feller nonlocal diffusion, which include \eqref{eq:fractional-burgers}, Achleitner and Ueda, \cite{AU} have established the local asymptotic stability of the traveling waves, with explicit algebraic decay rates for the perturbations. Their approach is based on a Lyapunov functional in a Sobolev setting, and it requires the initial perturbation to be small. In contrast, Theorem \ref{thm:optimality} concerns the case $u_-<u_+$, and it allows for large and non-integrable perturbations in $L^{q_c,\infty}(\R)$.

\smallskip
Rarefaction-wave convergence has also been studied for nonlocal diffusion generated by convolution kernels. Kawashima and Tanaka~\cite[Main Theorem, p.~212]{KS} proved stability for a radiating-gas system under smallness and regularity assumptions on the wave and perturbation. Pude\l{}ko~\cite[Theorem~2.1]{Pudelko2014} considered $u_t=J*u-u-uu_x$, where $J$ is a nonnegative even probability density with finite second moment. Assuming $u_0-w^R(0,\cdot)\in L^1(\R)$ and $u_{0,x}\in L^1(\R)$ with $u_{0,x}\ge0$, she obtained, for $p\in[1,\infty]$ and $t>0$,
\begin{equation}\label{eq:kernel-rarefaction-decay}
\|u(t,\cdot)-w^R(t,\cdot)\|_{L^p(\R)}
\le C_p t^{-\frac12(1-1/p)}[\log(2+t)]^{\frac12(1+1/p)},
\end{equation}
without a smallness assumption. The comparison with a viscous rarefaction profile is common to that analysis and ours. The fractional Laplacian has a different low-frequency order, so \eqref{eq:kernel-rarefaction-decay} does not provide the rate for \eqref{eq:fractional-burgers}.

\smallskip
For integrable initial data, Biler, Karch, and Woyczy\'nski~\cite[Corollary~3.2 and Theorem~4.1]{BKW2001} established $L^p$ decay and linear diffusion asymptotics for a class of conservation laws driven by L\'evy diffusion, under appropriate semigroup and flux assumptions. That finite-mass setting differs from the connection between two distinct end states considered here. More recently, Jakubowski and Serafin~\cite{JS2023} constructed solutions with singular, nonintegrable initial data for multidimensional fractional Burgers equations with power-law fluxes. In particular, they treated homogeneous data $M|x|^{-\beta}$ with $\beta\in(1,d)$ in dimension $d\ge2$ and established self-similarity and two-sided estimates. Their setting illustrates the role of scaling for data outside the usual Lebesgue classes; their dimension and flux assumptions differ from our one-dimensional quadratic equation and bounded perturbations of a rarefaction wave.

\smallskip
The bounded-data existence theory underlying our analysis was developed by Droniou, Gallou\"et, and Vovelle~\cite[Theorem~1.1]{DGV2003}, who established global solutions and smoothing for fractional regularizations of scalar hyperbolic equations, and by Droniou and Imbert~\cite[Theorem~7]{DI2006}, who also developed integral representations, maximum principles, and entropy inequalities. For diffusion of order at most one, Alibaud~\cite{Alibaud2007} introduced an entropy formulation and proved well-posedness for bounded initial data.

\subsection{Approach and further estimates} \label{approach-sec} We compare $u$ with a smooth monotone profile obtained from the classical viscous Burgers equation. Let $V$ solve the Riemann problem
\[
\left\{
\begin{aligned}
V_t-V_{xx}+VV_x&=0,
&& (t,x)\in(0,\infty)\times\R,\\
V(0,x)&=w^R(0,x),
&& x\in\R,
\end{aligned}
\right.
\qquad
w^R(0,x)=
\begin{cases}
u_-,&x<0,\\
u_+,&x\ge0.
\end{cases}
\]
We set $w(t,x)=V(t+1,x)$ so that the comparison profile is smooth at $t=0$. Since the viscous equation is invariant under time translation, $w$ satisfies 
\begin{equation} \label{w-eqn-oct-11}
w_t-w_{xx}+ww_x=0 \quad (t,x)\in(0,\infty)\times\R. 
\end{equation}
Its monotonicity follows from Theorem~\ref{thm:viscous-profile}, and its derivative estimates are stated in \eqref{eq:shifted-viscous-bounds}. By Corollary~\ref{cor:shifted-viscous-profile},
\[
w(0, \cdot)-w^R(0, \cdot)\in L^1(\R)\cap L^\infty(\R).
\]
Thus the initial condition in Theorem~\ref{thm:optimality} is equivalent to $z(0, \cdot)\in L^p(\R)\cap L^{q_c,\infty}(\R)$ for the perturbation $z=u-w$.

\smallskip
Subtracting \eqref{w-eqn-oct-11} from \eqref{eq:fractional-burgers} and setting $z=u-w$, we obtain
\begin{equation} \label{F-definition}
z_t+\Lambda^\alpha z+\Big(wz+\frac12z^2\Big)_x=F \quad \text{in} \quad (0,\infty) \times \R, \quad \text{where}
\quad F=-\bigl(\Lambda^\alpha w+w_{xx}\bigr).
\end{equation}
The central estimate is
\[
\|F(t, \cdot)\|_{L^{q_c,\infty}(\R)}\le C(1+t)^{-1},
\]
which is proved in \eqref{eq:residual-critical} below using \eqref{eq:shifted-viscous-bounds} and the weak-type Riesz-potential estimate; see \cite[Chapter~V, \S1, pp.~117--121]{Stein1970}. Using the monotonicity of $w$ in convex truncation estimates, we then obtain logarithmic control of the weak critical norm of $z$. The logarithmic factor in the upper bound arises from integrating the residual estimate in time.

\smallskip
For $p \in (q_c,\infty)$, we combine this control with fractional interpolation in a time-weighted energy estimate. For $p=\infty$, we use the one-sided slope bound $u_x\le1/t$ of Oleinik type \cite{Oleinik1957}. Since $w_x\ge0$, we have $z_x\le1/t$. Integrating this inequality over an interval and applying the weak critical bound, we obtain the supremum estimate directly. We approximate the initial perturbation by spatial truncation and convolution using Lemma~\ref{lem:approximation}, and pass to the limit in the estimates.

\smallskip 
Note that the additional $L^1$ assumption is not needed in Theorem~\ref{thm:optimality}, and we do not obtain decay from these bounds for $p \in [1,q_c]$, which is also consistent with \eqref{eq:optimality-lower}. In fact, for integrable initial perturbations, the same residual and truncation estimates also yield bounds at and below the critical exponent as in the following remark.
\begin{remark}\label{rem:below-critical}
If $z(0, \cdot)\in L^1(\R) \cap L^{q_c,\infty}(\R)$, then
\[
\|z(t, \cdot)\|_{L^1(\R)}\le C(1+t)^{2-\alpha}.
\]
In addition,
\[
\|z(t, \cdot)\|_{L^{q_c,\infty}(\R)}\le C\bigl[1+\log(2+t)\bigr]
\]
and
\[
\|z(t, \cdot)\|_{L^p(\R)}
\le C_p(1+t)^{\frac1p-(\alpha-1)}
\bigl[1+\log(2+t)\bigr]^{\frac{1-1/p}{2-\alpha}},
\]
for $p \in (1,q_c)$.
\end{remark}

\smallskip
\subsection{Paper organization} In Section~\ref{sec:critical}, we recall the existence theory, viscous-profile bounds, and interpolation inequalities, and derive the residual and critical estimates. In Section~\ref{sec:decay}, we prove the decay estimates, the bounds in Remark~\ref{rem:below-critical}, and the optimality assertion in Remark~\ref{optimal-remark}. In Section~\ref{sec:applications}, we apply the main theorem to integrable perturbations and to planar rarefaction waves in two space dimensions.

\section{Preliminary analysis and estimates}\label{sec:critical}

\subsection{Existence and smoothing}

We recall the bounded-data existence theorem of \cite[Theorem 1.1]{DGV2003} and \cite[Theorem 7]{DI2006}, in the form recorded in \cite[Theorem 2.1]{KMX2008}.

\begin{theorem}\label{thm:existence}
Let $\alpha \in (1,2)$ and $u_0\in L^\infty(\R)$. The Cauchy problem \eqref{eq:fractional-burgers} has a unique global bounded solution $u$. For every $T \in (0,\infty)$ and $\tau \in (0,T)$,
\[
u\in C_b((0,T]\times\R)\cap C_b^\infty([\tau,T]\times\R),
\qquad
\|u(t, \cdot)\|_{L^\infty(\R)}\le\|u_0\|_{L^\infty(\R)}.
\]
The equation holds classically on $(0, \infty) \times \R$. As $t\downarrow0$, we have $u(t, \cdot)\rightharpoonup^*u_0$ in $L^\infty(\R)$ and $u(t, \cdot)\to u_0$ in $L^r_{\mathrm{loc}}(\R)$ for every $r \in [1,\infty)$.
\end{theorem}

\smallskip
As discussed in Section \ref{approach-sec}, $w(t,x) = V(t+1, x)$  for $t \geq 0$ and $x \in \R$, where $V$ is defined in Theorem \ref{thm:viscous-profile} below. We use the following approximation by spatial truncation and convolution.
\begin{lemma}\label{lem:approximation}
Let $u_0\in L^\infty(\R)$, let $u$ be the solution in Theorem~\ref{thm:existence}, and set $z=u-w$. Let $\rho\in C_c^\infty(\R)$ be nonnegative with $\int_\R\rho=1$, and let $\chi\in C_c^\infty((-2,2))$ satisfy $0\le\chi\le1$ and $\chi=1$ on $[-1,1]$. We set
\begin{equation}\label{eq:approximation-data}
u_{0,n}=w(0,\cdot)+\rho_{1/n}*\bigl[\chi(\cdot/n)(u_0-w(0,\cdot))\bigr],
\qquad \rho_\varepsilon(x)=\varepsilon^{-1}\rho(x/\varepsilon).
\end{equation}
Then $u_{0,n}$ are smooth and uniformly bounded, and $u_{0,n}-w(0,\cdot)\in C_c^\infty(\R)$. The solutions $u_n$ and $u$ of \eqref{eq:fractional-burgers} with initial data $u_{0,n}$ and $u_0$, respectively, satisfy, for each $t>0$,
\begin{equation}\label{eq:approximation-local-convergence}
u_n(t,\cdot)\to u(t,\cdot)
\quad\text{in }L^1_{\mathrm{loc}}(\R).
\end{equation}
If also $u_0-w(0,\cdot)\in L^{q_c,\infty}(\R)$, then, for $p\in(q_c,\infty]$,
\begin{equation}\label{eq:approximation-initial-bounds}
\begin{aligned}
\|u_{0,n}-w(0,\cdot)\|_{L^p(\R)}
&\le\|u_0-w(0,\cdot)\|_{L^p(\R)},\\
\|u_{0,n}-w(0,\cdot)\|_{\weakL{q_c}}
&\le C_{q_c}\|u_0-w(0,\cdot)\|_{\weakL{q_c}}.
\end{aligned}
\end{equation}
Under this additional assumption, for every $r\in(q_c,\infty)$ and $T>0$,
\begin{equation}\label{eq:approximation-convergence}
\sup_{t \in [0,T]}\|u_n(t, \cdot)-u(t, \cdot)\|_{L^r(\R)}
\le C_T\|u_{0,n}-u_0\|_{L^r(\R)}\longrightarrow0.
\end{equation}
Moreover,
\begin{equation}\label{eq:approximation-lr-bounds}
\sup_{s\in[0,T]}\Bigl(\|z(s,\cdot)\|_{L^r(\R)}
+\sup_n\|u_n(s,\cdot)-w(s,\cdot)\|_{L^r(\R)}\Bigr)\le C_T.
\end{equation}
The constants $C_T>0$ and $C_{q_c}>0$ are independent of $n$.
\end{lemma}

\begin{proof}
By Young's inequality and Theorem~\ref{thm:existence},
\[
\begin{aligned}
\sup_{t\ge0}\max\Bigl\{\|u(t,\cdot)\|_{L^\infty(\R)},
\sup_n\|u_n(t,\cdot)\|_{L^\infty(\R)}\Bigr\}
&\le\max\Bigl\{\|u_0\|_{L^\infty(\R)},\sup_n\|u_{0,n}\|_{L^\infty(\R)}\Bigr\}\\
&\le\|w(0,\cdot)\|_{L^\infty(\R)}+\|z(0,\cdot)\|_{L^\infty(\R)}.
\end{aligned}
\]
For each $R>0$ and all sufficiently large $n$, $\chi((x-y)/n)=1$ when $|x|\le R$ and $y\in\operatorname{supp}\rho_{1/n}$. Hence
\[
\|u_{0,n}-u_0\|_{L^1(-R,R)}
=\|\rho_{1/n}*z(0,\cdot)-z(0,\cdot)\|_{L^1(-R,R)}\longrightarrow0.
\]
This uniform bound and the bounded-data stability results in \cite{DGV2003,DI2006} imply \eqref{eq:approximation-local-convergence}.

\smallskip
We now assume $z(0,\cdot)\in\weakL{q_c}$. By Young's inequality and $0\le\chi\le1$, we obtain
\[
\|u_{0,n}-w(0,\cdot)\|_{L^p(\R)}
\le\|\chi(\cdot/n)z(0,\cdot)\|_{L^p(\R)}
\le\|z(0,\cdot)\|_{L^p(\R)}.
\]
For every measurable set $E$ with $0<|E|<\infty$, we use \eqref{eq:weak-set-estimate} to obtain
\[
\begin{aligned}
\int_E|u_{0,n}(x)-w(0,x)|\,dx
&\le\int_\R\rho_{1/n}(y)\int_{E-y}|z(0,x)|\,dx\,dy\\
&\le\frac{q_c}{q_c-1}\|z(0,\cdot)\|_{\weakL{q_c}}|E|^{1-1/q_c}.
\end{aligned}
\]
Taking $E=\{x\in[-R,R]:|u_{0,n}(x)-w(0,x)|>\lambda\}$ and letting $R\to\infty$, we obtain the second bound in \eqref{eq:approximation-initial-bounds}.

\smallskip
For $r\in(q_c,\infty)$, we compute
\[
\begin{aligned}
\|z(0,\cdot)\|_{L^r(\R)}^r
&=r\int_0^\infty\lambda^{r-1}|\{x\in\R:|z(0,x)|>\lambda\}|\,d\lambda\\
&\le\frac{r}{r-q_c}\|z(0,\cdot)\|_{\weakL{q_c}}^{q_c}
\|z(0,\cdot)\|_{L^\infty(\R)}^{r-q_c}<\infty.
\end{aligned}
\]
By dominated convergence and approximation by convolution,
\[
\|u_{0,n}-u_0\|_{L^r(\R)}
\le\|[\chi(\cdot/n)-1]z(0,\cdot)\|_{L^r(\R)}+\|\rho_{1/n}*z(0,\cdot)-z(0,\cdot)\|_{L^r(\R)}
\longrightarrow0.
\]
On the other hand, for $0\le s\le t\le T$, we have
\[
\begin{aligned}
u_n(t,\cdot)-u(t,\cdot)
&=e^{-(t-s)\Lambda^\alpha}(u_n(s,\cdot)-u(s,\cdot))\\
&\quad-\frac12\int_s^t\partial_x e^{-(t-\tau)\Lambda^\alpha}
\bigl[(u_n+u)(u_n-u)\bigr](\tau,\cdot)\,d\tau.
\end{aligned}
\]
By \cite[eq.~(2.4)]{DGV2003},
\[
\|\partial_x e^{-\tau\Lambda^\alpha}\|_{L^q(\R)\to L^q(\R)}
\le C_\alpha\tau^{-1/\alpha}, \quad \text{for} \quad q\in\{r,\infty\}.
\]
Here we may take
\[
C=\frac{C_\alpha}{2}\sup_n\sup_{\tau\in[0,T]}
\|u_n(\tau,\cdot)+u(\tau,\cdot)\|_{L^\infty(\R)}<\infty.
\]
Since $\alpha>1$, we choose $\delta>0$ sufficiently small that
\[
C\int_0^\delta\tau^{-1/\alpha}\,d\tau
=\frac{C\alpha}{\alpha-1}\delta^{1-1/\alpha}\le\frac12.
\]
Starting with $s=0$, on each interval $[s,\min\{s+\delta,T\}]$ we solve the linear integral equation by contraction in $L^\infty_t(L^r_x\cap L^\infty_x)$. Uniqueness in $L^\infty_{t,x}$ identifies this solution with $u_n-u$. Thus the difference belongs to $L^r(\R)$, and
\[
\begin{aligned}
\|u_n(t,\cdot)-u(t,\cdot)\|_{L^r(\R)}
&\le\|u_n(s,\cdot)-u(s,\cdot)\|_{L^r(\R)}\\
&\quad+C\int_s^t(t-\tau)^{-1/\alpha}
\|u_n(\tau,\cdot)-u(\tau,\cdot)\|_{L^r(\R)}\,d\tau.
\end{aligned}
\]
Taking the supremum over $t\in[s,\min\{s+\delta,T\}]$, we obtain
\[
\begin{aligned}
&\sup_{t\in[s,\min\{s+\delta,T\}]}\|u_n(t,\cdot)-u(t,\cdot)\|_{L^r(\R)}\\
&\quad\le\|u_n(s,\cdot)-u(s,\cdot)\|_{L^r(\R)}
+\frac12\sup_{t\in[s,\min\{s+\delta,T\}]}\|u_n(t,\cdot)-u(t,\cdot)\|_{L^r(\R)}.
\end{aligned}
\]
Subtracting the last term and multiplying by two, we obtain
\[
\sup_{t\in[s,\min\{s+\delta,T\}]}\|u_n(t,\cdot)-u(t,\cdot)\|_{L^r(\R)}
\le2\|u_n(s,\cdot)-u(s,\cdot)\|_{L^r(\R)}.
\]
Iterating over the intervals, we conclude
\[
\sup_{t\in[0,T]}\|u_n(t,\cdot)-u(t,\cdot)\|_{L^r(\R)}
\le2^{\lceil T/\delta\rceil}\|u_{0,n}-u_0\|_{L^r(\R)}\longrightarrow0,
\]
which proves \eqref{eq:approximation-convergence}.

\smallskip
Note also that by Theorems~\ref{thm:existence} and~\ref{thm:viscous-profile} below, we have
\begin{equation}\label{eq:bounded-profile-perturbation}
\|z(t,\cdot)\|_{L^\infty(\R)}
\le\|u_0\|_{L^\infty(\R)}+\max\{|u_-|,|u_+|\},\qquad t\ge0.
\end{equation}
Since $u_{0,1}-w(0,\cdot)$ is compactly supported, Lemma~\ref{lem:perturbation-one} and \eqref{eq:bounded-profile-perturbation} imply
\[
\begin{aligned}
\|z(s,\cdot)\|_{L^r(\R)}
&\le\|u(s,\cdot)-u_1(s,\cdot)\|_{L^r(\R)}\\
&\quad+\|u_1(s,\cdot)-w(s,\cdot)\|_{L^1(\R)}^{1/r}
\|u_1(s,\cdot)-w(s,\cdot)\|_{L^\infty(\R)}^{1-1/r}
\le C_T.
\end{aligned}
\]
Combining this bound with \eqref{eq:approximation-convergence}, we obtain \eqref{eq:approximation-lr-bounds}.
\end{proof}

\smallskip
Now, we note that $\Lambda^\alpha=(-\partial_x^2)^{\alpha/2}$ is understood in the Fourier multiplier with symbol $|\xi|^\alpha$ for $\alpha \in (1,2)$ as stated in the introduction. However, in a few steps in our proofs below, we use the following integral presentations
\begin{equation} \label{fractional-def}
\begin{split}
\Lambda^\alpha f(x) & = c_\alpha \lim_{\epsilon \rightarrow 0^+}\int_{|h|>\epsilon} \frac{f(x) - f(x+h)}{|h|^{1+\alpha}} dh, \ \quad \text{and} \\\Lambda^\alpha f(x)  & = c_\alpha \lim_{\epsilon \rightarrow 0^+}\int_{\epsilon}^\infty \frac{2f(x) - f(x+h) - f(x-h)}{|h|^{1+\alpha}} dh,
\end{split}
\end{equation}
for $f \in C^2(\R)$ satisfying $(1+|x|)^{-1-\alpha} f (x)\in L^1(\R)$, where $c_\alpha>0$ is the normalization constant. See \cite[Theorem 1.1, Definition 2.7, and Lemma 3.2]{MK} for example. Below, we record a linear comparison lemma and its consequence for the fractional Burgers equation. Similar results can be found in \cite[Theorem 2 and Proposition 3]{DI2006}. 
\begin{lemma}\label{lem:linear-comparison}
Let $\alpha\in(1,2]$, $T>0$, and let $b,c\in C((0,T]\times\R)$ be bounded. Suppose that $h\in C([0,T]\times\R)\cap C^{1,2}((0,T]\times\R)$ is bounded and satisfies
\begin{equation}\label{eq:linear-comparison-hypothesis}
h_t+\Lambda^\alpha h+b h_x+ch\le0
\quad\text{in }(0,T]\times\R,
\quad \text{and} \quad h(0,x)\le0 \ \forall \ x \in \R.
\end{equation}
Then $h(t,x)\le0$ on $[0,T]\times\R$. Consequently, if $h_1,h_2 \in  C([0,T]\times\R)\cap C^{1,2}((0,T]\times\R)$, bounded, and satisfy
\begin{equation}\label{eq:linear-sub-super-comparison}
\begin{aligned}
(h_1)_t+\Lambda^\alpha h_1+b(h_1)_x+ch_1
&\le(h_2)_t+\Lambda^\alpha h_2+b(h_2)_x+ch_2 \ \text{in }(0,T]\times\R, \\
h_1(0,x)&\le h_2(0,x), \ \forall \ x \in \R,
\end{aligned}
\end{equation}
then $h_1\le h_2$.
More generally, if $h_t+\Lambda^\alpha h+bh_x+ch\le f$, where $f\in C([0,T];L^\infty(\R))$ is continuous in space, then
\begin{equation}\label{eq:linear-forced-upper-bound}
h(t,x)\le e^{\kappa t}\|h(0,\cdot)^+\|_{L^\infty(\R)}
+\int_0^t e^{\kappa(t-s)}\|f(s,\cdot)^+\|_{L^\infty(\R)}\,ds,
\end{equation}
where $\kappa=\|c^-\|_{L^\infty((0,T]\times\R)}$ and $r^\pm=\max\{\pm r,0\}$.
\end{lemma}

\begin{proof} We only need to consider the case $\alpha \in (1,2)$ as the case $\alpha =2$ follows the similar but simplier argument because $\Lambda^2 = -\partial_{xx}$. Replacing $h$ by $e^{-\kappa t}h$ and $c$ by $c+\kappa$, with $\kappa=\|c^-\|_{L^\infty((0,T]\times\R)}$, we may assume $c\ge0$. We set $\psi(x)=\sqrt{1+x^2}$. Since $\|\psi'\|_\infty\le1$ and $\|\psi''\|_\infty\le1$, we use \eqref{fractional-def} and apply the spherical corrodinates to obtain
\[
\begin{aligned}
|\Lambda^\alpha\psi(x)|
&\le c_\alpha\int_0^1 r^{1-\alpha}\,dr
+2c_\alpha\int_1^\infty r^{-\alpha}\,dr\\
&=c_\alpha\Bigl(\frac1{2-\alpha}+\frac2{\alpha-1}\Bigr), \quad \forall \ x \in \R.
\end{aligned}
\]
We choose $A>\|\Lambda^\alpha\psi\|_\infty+\|b\|_\infty$ and, for $\eta>0$, we define
\[
h_\eta(t,x)=h(t,x)-\eta\bigl(\psi(x)+At\bigr) \quad (t,x)\in[0,T]\times\R.
\]
Then 
\begin{equation}  \label{h-eta-initial}
h_\eta(0, \cdot) \leq -\eta \ , \ \text{and} \quad h_\eta(t,x)\to-\infty  \ \text{ as } \ |x|\to\infty
\ \text{ uniformly for } \ t\in[0,T].
\end{equation} 
Moreover,
\begin{align} \notag
(h_\eta)_t+\Lambda^\alpha h_\eta+b(h_\eta)_x+ch_\eta
&\le-\eta\bigl[A+\Lambda^\alpha\psi+b\psi'+c(\psi+At)\bigr]\\ \label{h-eta-eqn}
&\le-\eta\bigl[A-\|\Lambda^\alpha\psi\|_\infty-\|b\|_\infty\bigr]<0 \quad \text{in }(0,T]\times\R.
\end{align}
Now, if $h_\eta$ is non-negative somewhere, due to \eqref{h-eta-initial} and the continuity of $h_\eta$ on compact sets, we can find a first contact point $(t_*,x_*)$ with $t_*>0$ such that $h_\eta(t_*,x_*) =0$. Then, it follows that
\[
(h_\eta)_t(t_*,x_*)\ge0, \quad \text{and} \quad
(h_\eta)_x(t_*,x_*)=0,
\]
Moreover, as in \eqref{fractional-def}, we have
\[
\Lambda^\alpha h_\eta(t_*,x_*)
=c_\alpha\, \int_\R
\frac{-h_\eta(t_*,y)}{|x_*-y|^{1+\alpha}}\,dy\ge0
\]
This contradicts \eqref{h-eta-eqn}. Hence $h_\eta<0$. Letting $\eta\downarrow0$, we conclude $h\le0$.

\smallskip
Applying the result we just proved to $h_1-h_2$, we obtain \eqref{eq:linear-sub-super-comparison}. On the other hand, to prove \eqref{eq:linear-forced-upper-bound}, we set
\[
B(t)=e^{\kappa t}\|h(0,\cdot)^+\|_{L^\infty(\R)}
+\int_0^t e^{\kappa(t-s)}\|f(s,\cdot)^+\|_{L^\infty(\R)}\,ds.
\]
Then  $B\ge0$, $h(0,x)\le B(0)$, and
\[
B'(t)+c(t,x)B(t)
=(\kappa+c(t,x))B(t)+\|f(t,\cdot)^+\|_{L^\infty(\R)}
\ge f(t,x).
\]
Since $B$ is independent of $x$, we have $\Lambda^\alpha B=B_x=0$, and thus
\[
(h-B)_t+\Lambda^\alpha(h-B)+b(h-B)_x+c(h-B)\le0,
\quad \text{and} \quad (h-B)(0,x)\le0.
\]
Applying the comparison principle that we just proved, we obtain $h\le B$, which is \eqref{eq:linear-forced-upper-bound}. The proof of the lemma is complete.
\end{proof}

\smallskip
We apply Lemma \ref{lem:linear-comparison} to the difference of two solutions to derive the following corollary.
\begin{corollary}\label{cor:burgers-comparison}
Let $\alpha\in(1,2]$ and $T>0$. The following assertions hold.
\begin{itemize}
\item[\textup{(i)}] Suppose that $u,v\in C([0,T]\times\R)\cap C^{1,2}((0,T]\times\R)$ 
are bounded, $v_x\in L^\infty((0,T]\times\R)$, and
\[
\begin{aligned}
u_t+\Lambda^\alpha u+uu_x
&\le v_t+\Lambda^\alpha v+vv_x
&&\text{in }(0,T]\times\R,\\
u(0,x)&\le v(0,x)
&&\text{for }x\in\R.
\end{aligned}
\]
Then $u(t,x)\le v(t,x)$ on $[0,T]\times\R$.

\smallskip
\item[\textup{(ii)}] If $u_0,v_0\in L^\infty(\R)$ and $u,v$ are the corresponding solutions of \eqref{eq:fractional-burgers} with initial data $u_0,v_0$ satisfying $u_0\le v_0$ almost everywhere, then 
\[ u(t,x)\le v(t,x)\quad\text{for }t>0,\ x\in\R.
\]
In particular, for every $t>0$ and $x\in\R$,
\begin{equation}\label{eq:burgers-range-bound}
\operatorname*{ess\,inf}_{y\in\R}u_0(y)
\le u(t,x)\le
\operatorname*{ess\,sup}_{y\in\R}u_0(y).
\end{equation}
Moreover, if $u_0$ has a nondecreasing representative, then $u(t,\cdot)$ is nondecreasing for every $t>0$.
\end{itemize}
\end{corollary}

\begin{proof}
Using $uu_x-vv_x=u(u-v)_x+v_x(u-v)$, we obtain
\[
(u-v)_t+\Lambda^\alpha(u-v)+u(u-v)_x+v_x(u-v)\le0.
\]
We apply Lemma~\ref{lem:linear-comparison} with $h=u-v$, $b=u$, and $c=v_x$ to conclude $u\le v$.

\smallskip
For bounded initial data, we use Lemma~\ref{lem:approximation} with the same $w(0,\cdot)$, $\rho$, and $\chi$ for both data. Then
\[
u_{0,n}-v_{0,n}
=\rho_{1/n}*\bigl[\chi(\cdot/n)(u_0-v_0)\bigr]\le0.
\]
For these smooth data, we apply part~\textup{(i)} to obtain $u_n\le v_n$. By \eqref{eq:approximation-local-convergence}, we pass to the limit and obtain $u(t,\cdot)\le v(t,\cdot)$ almost everywhere, hence everywhere by continuity.

\smallskip
Comparing $u$ with the constant solutions $\operatorname*{ess\,inf}_\R u_0$ and $\operatorname*{ess\,sup}_\R u_0$, we obtain \eqref{eq:burgers-range-bound}. If $u_0$ is nondecreasing, then for every $a>0$,
\[
u_0(x)\le u_0(x+a)\quad\text{a.e. in }\R.
\]
Since $u(t,x+a)$ solves the same equation with initial datum $u_0(x+a)$, we obtain $u(t,x)\le u(t,x+a)$ by comparison. Thus $u(t,\cdot)$ is nondecreasing.
\end{proof}
\subsection{The viscous rarefaction profile}

We first recall the classical viscous-profile estimates derived in \cite{HN1991} and \cite[Propositions~3.2--3.3, p.~224]{KS},   using the profile approximation only for $p \in (1,\infty]$. See also \cite[Lemma~2.4]{KMX2008} for the summary of the results. The initial datum in this result is exactly the Riemann step, rather than an arbitrary function with the same end states.
\begin{theorem}\label{thm:viscous-profile}
Let $u_-<u_+$ and $w^R$ be the rarefaction wave \eqref{eq:rarefaction} with initial datum \eqref{w-R0}. The problem
\[
\left\{
\begin{array}{cccl}
V_t-V_{xx}+VV_x & = &0 & \quad \text{in}\quad (0, \infty) \times \R,\\
V(0, \cdot) & = & w^R(0,\cdot) & \quad \text{in} \quad \R,
\end{array} \right.
\]
has a unique global bounded solution $V$, smooth for positive time, satisfying $u_-<V(t,x)<u_+$ and $V_x(t,x)>0$ for $t>0$ and $x\in\R$. Moreover, for every $p \in [1,\infty]$ and $t>0$,
\begin{equation}\label{eq:viscous-derivatives}
\|V_x(t, \cdot)\|_{L^p(\R)}\le C_p t^{-1+1/p},\qquad
\|V_{xx}(t, \cdot)\|_{L^p(\R)}\le C_p t^{-3/2+1/(2p)},
\end{equation}
and, for $p \in (1,\infty]$,
\begin{equation}\label{eq:viscous-approximation}
\|V(t, \cdot)-w^R(t, \cdot)\|_{L^p(\R)}\le C_p t^{-\frac12(1-1/p)}
\end{equation}
where $C_p =C(p, u_+, u_-) >0$.
\end{theorem}
\noindent
We next record the corresponding estimates for the time-shifted profile and verify its initial integrability relative to \eqref{w-R0}.
\begin{corollary}\label{cor:shifted-viscous-profile} Let $V$ be defined in Theorem \ref{thm:viscous-profile}, and set $w(t,x)=V(t+1,x)$ for $x \in \R$ and $t \geq 0$. Then, for $p \in [1,\infty]$ and $t\ge0$,
\begin{equation}\label{eq:shifted-viscous-bounds}
\begin{aligned}
& \|w_x(t, \cdot)\|_{L^p(\R)}\le C_p(1+t)^{-1+1/p},\\
& \|w_{xx}(t, \cdot)\|_{L^p(\R)} \le C_p(1+t)^{-3/2+1/(2p)},
\end{aligned}
\end{equation}
and, for $p \in (1,\infty]$,
\begin{equation}\label{eq:shifted-profile-approximation}
\|w(t, \cdot)-w^R(t, \cdot)\|_{L^p(\R)} \le C_p(1+t)^{-\frac12(1-1/p)}, \quad \forall  \ t \ge 0,
\end{equation}
where $C_p = C(p, u_+, u_-)>0$. Moreover, $w(0, \cdot)-w^R(0,\cdot)\in L^1(\R)\cap L^\infty(\R)$. 
\end{corollary}

\begin{proof}
Applying \eqref{eq:viscous-derivatives} at time $t+1$, we obtain \eqref{eq:shifted-viscous-bounds}. To compare $w(t, \cdot)$ with $w^R(t, \cdot)$ for $p \in (1,\infty]$, we use
\begin{equation} \label{w-t-est-oct-12}
\|w(t, \cdot)-w^R(t, \cdot)\|_{L^p}
\le\|V(t+1, \cdot)-w^R(t+1, \cdot)\|_{L^p}
+\|w^R(t+1, \cdot)-w^R(t, \cdot)\|_{L^p}.
\end{equation}
Using the explicit formula \eqref{eq:rarefaction}, we obtain
\[
\|\partial_t w^R(t, \cdot)\|_{L^p}\le C_p t^{-1+1/p},\qquad t>0.
\]
For $t\ge1$, we therefore have
\[
\begin{aligned}
\|w^R(t+1,\cdot)-w^R(t,\cdot)\|_{L^p}
&\le\int_t^{t+1}\|\partial_s w^R(s,\cdot)\|_{L^p}\,ds\\
&\le C_p t^{-1+1/p}\le C_p t^{-\frac12(1-1/p)}.
\end{aligned}
\]
For $t\in[0,1]$, this difference is uniformly bounded and supported in a fixed bounded interval. Combining these bounds with \eqref{w-t-est-oct-12} and \eqref{eq:viscous-approximation}, we obtain \eqref{eq:shifted-profile-approximation}.

\smallskip
To estimate $w(0,\cdot)-w^R(0,\cdot)$, we apply the heat-semigroup formula to obtain
\begin{equation} \label{V-1-est-oct-8}
V(1, \cdot)-w^R(0, \cdot)
=(e^{\partial_x^2}-I)w^R(0, \cdot)
-\int_0^1 e^{(1-s)\partial_x^2}\bigl(V(s, \cdot)V_x(s, \cdot)\bigr)\,ds.
\end{equation}
The first term on the right-hand side of \eqref{V-1-est-oct-8} is integrable because $w^R(0, \cdot)$ is a step function. On the other hand, the integral term in \eqref{V-1-est-oct-8} is bounded in $L^1$ by
\[
\sup_{s \in (0,1]}\|V(s, \cdot)\|_{L^\infty}
\int_0^1\|V_x(s, \cdot)\|_{L^1}\,ds
\le C,
\]
for some $C = C(u_0, u_+)>0$ Here we use $u_-\le V(s,x)\le u_+$ and $\|V_x(s,\cdot)\|_{L^1(\R)}\leq C$ for all $s \in [0,1]$ from Theorem~\ref{thm:viscous-profile}. Since $w^R(0,x)$ also lies in $[u_-,u_+]$, we have
\[
\|V(1,\cdot)-w^R(0,\cdot)\|_{L^\infty(\R)}\le u_+-u_-.
\] We have thus verified the assertion $w(0, \cdot)-w^R(0,\cdot)\in L^1(\R)\cap L^\infty(\R)$, and the proof is complete.
\end{proof}

\smallskip
We note that the restriction $p \in (1,\infty]$ in the profile approximation \eqref{eq:viscous-approximation} and similarly in \eqref{eq:shifted-profile-approximation} is essential. In \cite[Proposition~3.2 and the discussion following it, p.~224]{KS}, the authors prove the estimate for $p>1$ and give a logarithmic upper bound for $p=1$, while leaving the possibility of a uniform $L^1$ bound open. To see directly why a uniform bound fails, for the symmetric end states $u_-=-1$ and $u_+=1$, we obtain from the Cole--Hopf formula
\[
\begin{aligned}
\int_0^\infty\bigl(w^R(t,x)-V(t,x)\bigr)\,dx
&=-2\log\Bigl[\frac1{\sqrt{\pi t}}
\int_0^\infty e^{-y/2-y^2/(4t)}\,dy\Bigr]\\
&=\log t+O(1),\qquad t\to\infty.
\end{aligned}
\]
Thus $\|V(t, \cdot)-w^R(t, \cdot)\|_{L^1}$ cannot be bounded uniformly in time. Only the derivative estimates and the initial $L^1$ comparison are used at $p=1$ in this paper.

\subsection{Interpolation inequalities}

We use the following fractional Gagliardo--Nirenberg inequality; see \cite[Section~2.4]{Triebel1983}.
\begin{lemma}\label{lem:fractional-interpolation}
Let $\alpha \in (1,2)$ and $p \in [1,\infty]$. There is $C = C(\alpha, p)>0$ such that 
\begin{equation}\label{eq:fractional-interpolation}
\|\Lambda^\alpha f\|_{L^p(\R)}
\le C\|f_x\|_{L^p(\R)}^{2-\alpha}
\|f_{xx}\|_{L^p(\R)}^{\alpha-1},
\end{equation}
for every $f\in W^{2,p}_{\mathrm{loc}}(\R)$ such that $f_x,f_{xx}\in L^p(\R)$.
\end{lemma}

In the next inequality, we interpolate between an integrable norm and a weak-Lorentz norm. 
\begin{lemma}\label{lem:below-interpolation}
Let $p \in (1,\infty)$ and  $q \in (1,p)$. There is $C_{q,p}>0$ such that  
\begin{equation}\label{eq:below-interpolation}
\|f\|_{L^q(\R)}
\le C_{q,p}\|f\|_{L^1(\R)}^{\frac{p-q}{q(p-1)}}
\|f\|_{L^{p,\infty}(\R)}^{\frac{p(q-1)}{q(p-1)}}
\end{equation}
for every $f\in L^1(\R)\cap L^{p,\infty}(\R)$.
\end{lemma}

\begin{proof}We provide the proof for completeness.  Note that we only need to consider the case $f\ne0$. For every $\lambda>0$, by Chebyshev's inequality  and \eqref{Lorentz-norm-def} it follows that
\[
\lambda|\{|f|>\lambda\}|
\le\int_{\{|f|>\lambda\}}|f(x)|\,dx\le\|f\|_{L^1(\R)},
\quad \text{and} \quad
\lambda^p|\{|f|>\lambda\}|\le\|f\|_{L^{p,\infty}(\R)}^p.
\]
Thus, for every $\lambda_0>0$, we have
\[
\begin{aligned}
\|f\|_{L^q(\R)}^q
&=q\int_0^\infty\lambda^{q-1}|\{|f|>\lambda\}|\,d\lambda\\
&=q\int_0^{\lambda_0}\lambda^{q-1}|\{|f|>\lambda\}|\,d\lambda
+q\int_{\lambda_0}^\infty\lambda^{q-1}|\{|f|>\lambda\}|\,d\lambda\\
&\le q\|f\|_{L^1(\R)}\int_0^{\lambda_0}\lambda^{q-2}\,d\lambda
+q\|f\|_{L^{p,\infty}(\R)}^p\int_{\lambda_0}^\infty\lambda^{q-p-1}\,d\lambda\\
&=\frac{q}{q-1}\|f\|_{L^1(\R)}\lambda_0^{q-1}
+\frac{q}{p-q}\|f\|_{L^{p,\infty}(\R)}^p\lambda_0^{q-p}.
\end{aligned}
\]
We choose $\lambda_0=\Bigl(\frac{\|f\|_{L^{p,\infty}(\R)}^p}{\|f\|_{L^1(\R)}}\Bigr)^{1/(p-1)}$ and take the $q$-th root to obtain the stated estimate. 
\end{proof}

\smallskip
Above the critical exponent, we use the following interpolation estimates to control the perturbation and forcing terms by the fractional dissipation and the weak critical norm.

\begin{lemma}\label{interpolation}
Let $p \in (q_c,\infty)$ and assume
\[
f\in L^{q_c,\infty}(\R),
\quad \text{and} \quad \Lambda^{\alpha/2}|f|^{p/2}\in L^2(\R),
\]
where the second condition is understood in the homogeneous Sobolev sense. Then
\begin{equation}\label{eq:critical-sup-interpolation}
\|f\|_{L^\infty(\R)}^{p+1}\le C_p\Bigl\|\Lambda^{\alpha/2}|f|^{p/2}\Bigr\|_{L^2(\R)}^2\|f\|_{\weakL{q_c}}
\end{equation}
and
\begin{equation}\label{eq:critical-lp-interpolation}
\|f\|_{L^p(\R)}^p
\le C_p\Bigl\|\Lambda^{\alpha/2}|f|^{p/2}\Bigr\|_{L^2(\R)}^{\frac{2(p-q_c)}{p+1}}
\|f\|_{\weakL{q_c}}^{\frac{p(q_c+1)}{p+1}}.
\end{equation}
In particular, for every $\lambda,\varepsilon>0$,
\begin{equation}\label{eq:weighted-interpolation}
\lambda\|f\|_{L^p(\R)}^p
\le\varepsilon \Bigl\|\Lambda^{\alpha/2}|f|^{p/2}\Bigr\|_{L^2(\R)}^2+C_{p,\varepsilon}\lambda^{\frac{p+1}{q_c+1}}\|f\|_{\weakL{q_c}}^p.
\end{equation}
For $g\in\weakL{q_c}$, we also have
\begin{equation}\label{eq:forcing-interpolation}
\Bigl|\int_\R g|f|^{p-2}f\,dx\Bigr|
\le\varepsilon \Bigl\|\Lambda^{\alpha/2}|f|^{p/2}\Bigr\|_{L^2(\R)}^2+C_{p,\varepsilon}
\|g\|_{\weakL{q_c}}^{\frac{p+1}{q_c+1}}\|f\|_{\weakL{q_c}}^{\frac{pq_c-1}{q_c+1}}.
\end{equation}
Here $C_p$ and $C_{p,\varepsilon}$ may also depend on $\alpha$, which is fixed throughout.
\end{lemma}

\begin{proof}
If $\Lambda^{\alpha/2}|f|^{p/2}=0$, then $|f|^{p/2}$ is constant, and we infer $f=0$ from $f\in L^{q_c,\infty}(\R)$ and the assertions in the lemma are trivial. Otherwise, we use the continuous representative of $|f|^{p/2}$ and apply the homogeneous fractional Morrey inequality in $\dot H^{\alpha/2}(\R)$; see \cite[Corollary~2.7 and Theorem~4.4]{BrascoGomezCastroVazquez2021}. In terms of the equivalent Fourier seminorm, we obtain
\begin{equation} \label{Morrey-est}
\bigl||f(x)|^{p/2}-|f(y)|^{p/2}\bigr|
\le C \bigl\|\Lambda^{\alpha/2}|f|^{p/2}\bigr\|_{L^2(\R)}|x-y|^{(\alpha-1)/2}.
\end{equation}

\smallskip
We fix $x$ with $|f(x)|>0$. On an interval centered at $x$ of radius
\[
r=c_p\Bigl(\frac{|f(x)|^p}{\|\Lambda^{\alpha/2}|f|^{p/2}\|_{L^2(\R)}^2}\Bigr)^{1/(\alpha-1)},
\]
where $c_p>0$ is sufficiently small, we infer from \eqref{Morrey-est} that $|f(y)|>|f(x)|/2$ whenever $|y-x|<r$. Hence $(x-r,x+r)\subset\{y\in\R:|f(y)|>|f(x)|/2\}$ and then
\[
\bigl|\{y\in\R:|f(y)|>|f(x)|/2\}\bigr|\ge2r.
\]
Using the level $\lambda=|f(x)|/2$ in the definition of the weak-Lorentz quasi-norm \eqref{Lorentz-norm-def}, we obtain
\[
\begin{aligned}
\|f\|_{\weakL{q_c}}
&\ge\frac{|f(x)|}{2}\bigl|\{y\in\R:|f(y)|>|f(x)|/2\}\bigr|^{1/q_c}\\
&\ge\frac{|f(x)|}{2}(2r)^{1/q_c}.
\end{aligned}
\]
Since $1/q_c=\alpha-1$, it then follows from the choice of $r$ that
\[
\|f\|_{\weakL{q_c}}
\ge 2^{1/q_c-1}c_p^{1/q_c}
\frac{|f(x)|^{p+1}}{\bigl\|\Lambda^{\alpha/2}|f|^{p/2}\bigr\|_{L^2(\R)}^2}.
\]
Taking the supremum over $x$, we obtain \eqref{eq:critical-sup-interpolation}.

\smallskip
Next, by the definition of the weak-Lorentz quasi-norm \eqref{Lorentz-norm-def} again, we see that 
\begin{equation}\label{eq:weak-level-set}
|\{|f|>\lambda\}|\le\lambda^{-q_c}\|f\|_{\weakL{q_c}}^{q_c},
\end{equation}
for every $\lambda>0$. Using \eqref{eq:weak-level-set} and \eqref{eq:critical-sup-interpolation}, we see that
\[
\begin{aligned}
\|f\|_{L^p(\R)}^p
& =p\int_0^\infty\lambda^{p-1}|\{x\in\R:|f(x)|>\lambda\}|\,d\lambda=p\int_0^{\|f\|_{L^\infty(\R)}}\lambda^{p-1}|\{|f|>\lambda\}|\,d\lambda\\
&\le p\|f\|_{\weakL{q_c}}^{q_c}\int_0^{\|f\|_{L^\infty(\R)}}\lambda^{p-q_c-1}\,d\lambda=\frac{p}{p-q_c}\|f\|_{\weakL{q_c}}^{q_c}\|f\|_{L^\infty(\R)}^{p-q_c}\\
&\le C_p\bigl\|\Lambda^{\alpha/2}|f|^{p/2}\bigr\|_{L^2(\R)}^{\frac{2(p-q_c)}{p+1}}
\|f\|_{\weakL{q_c}}^{\frac{p(q_c+1)}{p+1}}.
\end{aligned}
\]
This proves \eqref{eq:critical-lp-interpolation}. Applying Young's inequality with conjugate exponents $\frac{p+1}{p-q_c}$ and $\frac{p+1}{q_c+1}$,
we obtain \eqref{eq:weighted-interpolation}.

\smallskip
For $q\in(1,\infty)$, $g\in L^{q,\infty}(\R)$, and $|E|<\infty$, we apply \eqref{eq:weak-level-set} with $f=g$ and $q_c$ replaced by $q$ to obtain
\begin{equation}\label{eq:weak-set-estimate}
\begin{aligned}
\int_E|g|\,dx
&=\int_0^\infty|E\cap\{x\in\R:|g(x)|>\lambda\}|\,d\lambda\\
&\le\int_0^\infty\min\Bigl\{|E|,\lambda^{-q}\|g\|_{L^{q,\infty}(\R)}^q\Bigr\}\,d\lambda
=\frac{q}{q-1}\|g\|_{L^{q,\infty}(\R)}|E|^{1-1/q}.
\end{aligned}
\end{equation}
Then, we apply \eqref{eq:weak-set-estimate} with $q=q_c$ to obtain
\[
\begin{aligned}
\Bigl|\int_\R g|f|^{p-2}f\,dx\Bigr|
&\le(p-1)\int_0^{\|f\|_{L^\infty(\R)}}\lambda^{p-2}
\int_{\{|f|>\lambda\}}|g|\,dx\,d\lambda\\
&\le C_p\|g\|_{\weakL{q_c}}\|f\|_{\weakL{q_c}}^{q_c-1}
\int_0^{\|f\|_{L^\infty(\R)}}\lambda^{p-q_c-1}\,d\lambda\\
&\le C_p\|g\|_{\weakL{q_c}}
\bigl\|\Lambda^{\alpha/2}|f|^{p/2}\bigr\|_{L^2(\R)}^{\frac{2(p-q_c)}{p+1}}
\|f\|_{\weakL{q_c}}^{q_c-1+\frac{p-q_c}{p+1}}.
\end{aligned}
\]
The integral converges because $p \in (q_c,\infty)$. Applying Young's inequality once more, we obtain \eqref{eq:forcing-interpolation}, since
\[
\frac{p+1}{q_c+1}
\Big(q_c-1+\frac{p-q_c}{p+1}\Big)
=\frac{pq_c-1}{q_c+1}.
\]
The proof is then complete.
\end{proof}

\subsection{Residual and critical estimates}

Throughout this subsection, $u$ is the bounded solution from Theorem~\ref{thm:existence}, $w(t, \cdot)=V(t+1, \cdot)$ is the profile from Corollary~\ref{cor:shifted-viscous-profile}, and $z=u-w$ satisfies \eqref{F-definition}. We begin by estimating the residual $F$ in that equation in the critical weak-Lorentz space.

\begin{lemma}\label{lem:residual-critical}
For all $t\ge0$,
\begin{equation}\label{eq:residual-critical}
\|F(t, \cdot)\|_{\weakL{q_c}}\le C(1+t)^{-1}.
\end{equation}
\end{lemma}

\begin{proof}
For fixed $t\ge0$ and $x\in\R$, we use Taylor's formula and boundedness to obtain, for $h>0$,
\[
\begin{aligned}
\frac{|2w(t,x)-w(t,x+h)-w(t,x-h)|}{h^\alpha}
&\le\begin{cases}
h^{2-\alpha}\|w_{xx}(t,\cdot)\|_{L^\infty(\R)},&0<h\le1,\\
4h^{-\alpha}\|w(t,\cdot)\|_{L^\infty(\R)},&h\ge1,
\end{cases}\\
\frac{|w_x(t,x+h)-w_x(t,x-h)|}{h^{\alpha-1}}
&\le\begin{cases}
2h^{2-\alpha}\|w_{xx}(t,\cdot)\|_{L^\infty(\R)},&0<h\le1,\\
2h^{1-\alpha}\|w_x(t,\cdot)\|_{L^\infty(\R)},&h\ge1.
\end{cases}
\end{aligned}
\]
Since $\alpha\in(1,2)$, both boundary terms tend to zero as $h\downarrow0$ and as $h\to\infty$. Then, by using \eqref{fractional-def} and then integrating by parts twice, we obtain
\begin{align} \notag
\Lambda^\alpha w(t,x)
&=c_\alpha\int_0^\infty
\frac{2w(t,x)-w(t,x+h)-w(t,x-h)}{h^{1+\alpha}}\,dh\\ \notag
&=-\frac{c_\alpha}{\alpha}\int_0^\infty
\frac{w_x(t,x+h)-w_x(t,x-h)}{h^\alpha}\,dh\\ \notag
&=-\frac{c_\alpha}{\alpha(\alpha-1)}\int_0^\infty
h^{1-\alpha}\bigl[w_{xx}(t,x+h)+w_{xx}(t,x-h)\bigr]\,dh\\ \label{Riesz-potential}
&=-\frac{c_\alpha}{\alpha(\alpha-1)}
\int_\R |x-y|^{1-\alpha}w_{xx}(t,y)\,dy,
\end{align}
where $c_\alpha>0$ is the normalization constant. The last integral is the Riesz-potential representation of $-w_{xx}$; see \cite[Chapter~V, \S1, pp.~117--121]{Stein1970}. It is absolutely convergent, since
\[
\int_\R |x-y|^{1-\alpha}|w_{xx}(t,y)|\,dy
\le\frac{2}{2-\alpha}\|w_{xx}(t,\cdot)\|_{L^\infty(\R)}
+\|w_{xx}(t,\cdot)\|_{L^1(\R)}<\infty.
\]

\smallskip
For every measurable set $E$ with $0<|E|<\infty$ and every $y\in\R$, we have
\[
\begin{aligned}
\int_E|x-y|^{1-\alpha}\,dx
&\le\int_{|x-y|\le |E|}|x-y|^{1-\alpha}\,dx
+\int_{E\cap\{|x-y|>|E|\}}|E|^{1-\alpha}\,dx\\
&\le 2\int_0^{|E|}r^{1-\alpha}\,dr+|E|^{1-\alpha}|E|\\
&=\Bigl(\frac{2}{2-\alpha}+1\Bigr)|E|^{2-\alpha}.
\end{aligned}
\]
Taking the supremum over $y\in\R$, we obtain
\[
\sup_{y\in\R}\int_E|x-y|^{1-\alpha}\,dx
\le C_\alpha |E|^{2-\alpha}.
\]
From this, \eqref{Riesz-potential}, and by applying Tonelli's theorem, we find
\[
\int_E|\Lambda^\alpha w(t,x)|\,dx
\le C_\alpha |E|^{2-\alpha}\|w_{xx}(t, \cdot)\|_{L^1(\R)}.
\]
For $\lambda,R>0$, we choose $E=\{x\in[-R,R]:|\Lambda^\alpha w(t,x)|>\lambda\}$. By the definition of $E$ and the preceding estimate, we obtain
\[
\lambda |E|\le\int_E|\Lambda^\alpha w(t,x)|\,dx
\le C_\alpha |E|^{2-\alpha}\|w_{xx}(t,\cdot)\|_{L^1(\R)}.
\]
If $|E|>0$, we divide by $|E|^{2-\alpha}$ to obtain
\[
\lambda |E|^{\alpha-1}\le C_\alpha\|w_{xx}(t,\cdot)\|_{L^1(\R)}.
\]
The same inequality is immediate when $|E|=0$.
Letting $R\to\infty$ and taking the supremum over $\lambda>0$, we use $1/q_c=\alpha-1$ and \eqref{eq:shifted-viscous-bounds} to conclude
\[
\|\Lambda^\alpha w(t, \cdot)\|_{\weakL{q_c}}
\le C\|w_{xx}(t, \cdot)\|_{L^1(\R)}\le C(1+t)^{-1}.
\]
Also, using the embedding $L^{q_c}(\R)\hookrightarrow L^{q_c,\infty}(\R)$ and \eqref{eq:shifted-viscous-bounds}, we obtain
\[
\|w_{xx}(t, \cdot)\|_{\weakL{q_c}}
\le C\|w_{xx}(t, \cdot)\|_{L^{q_c}(\R)}
\le C(1+t)^{-(4-\alpha)/2}\le C(1+t)^{-1}.
\]
From the definition of $F$ in \eqref{F-definition} and the last two estimates, we obtain the result. The proof is complete.
\end{proof}

We also need an $L^1(\R)$ bound on the residual to control the perturbation below the critical exponent.

\begin{lemma}\label{lem:residual-one}
For all $t\ge0$,
\begin{equation}\label{eq:residual-one}
\|F(t, \cdot)\|_{L^1(\R)}\le C(1+t)^{1-\alpha}.
\end{equation}
\end{lemma}

\begin{proof}
Applying \eqref{eq:fractional-interpolation} with $p=1$ and \eqref{eq:shifted-viscous-bounds}, we obtain
\[
\begin{aligned}
\|\Lambda^\alpha w(t,\cdot)\|_{L^1(\R)}
&\le C\|w_x(t,\cdot)\|_{L^1(\R)}^{2-\alpha}
\|w_{xx}(t,\cdot)\|_{L^1(\R)}^{\alpha-1}\\
&\le C(1+t)^{-(\alpha-1)}.
\end{aligned}
\]
Due to \eqref{eq:shifted-viscous-bounds},  it follows that $\|w_{xx}(t, \cdot)\|_{L^1(\R)}\le C(1+t)^{-1}\le C(1+t)^{1-\alpha}$. The same bound holds for $F$ by \eqref{F-definition}.
\end{proof}

Integrating \eqref{eq:residual-one} in time, we obtain an $L^1(\R)$ estimate for the perturbation.

\begin{lemma}\label{lem:perturbation-one}
If $z(0, \cdot)\in L^1(\R)$, then
\begin{equation}\label{eq:perturbation-one}
\|z(t, \cdot)\|_{L^1(\R)}\le C(1+t)^{2-\alpha},\qquad t\ge0.
\end{equation}
\end{lemma}

\begin{proof}
We first justify integrability of the perturbation. In \eqref{F-definition}, we write the flux as $(u+w)z/2$ and use the mild formulation with forcing $F$. On every finite time interval, the coefficient $(u+w)/2$ is bounded and $F\in L^1_t(L^1_x\cap L^\infty_x)$. Using the heat-kernel derivative bound $\|\partial_x e^{-t\Lambda^\alpha}\|_{L^1\to L^1}\le C_\alpha t^{-1/\alpha}$ from \cite[eq.~(2.4)]{DGV2003}, we construct the mild solution on successive short intervals in $L^1\cap L^\infty$. Uniqueness of the bounded mild solution identifies it with $z$. Thus $z\in C([0,T];L^1(\R))$ for every $T>0$.

\smallskip
We now test the perturbation equation with a smooth approximation of $\operatorname{sgn}(z)$ and spatial cutoffs, and then pass to the limit. Applying the fractional Kato inequality from \cite[Lemma~1]{DI2006}, we find that the dissipation is nonnegative. We also compute
\[
\int_\R(wz)_x\operatorname{sgn}(z)\,dx
=\int_\R(w|z|)_x\,dx=0,
\qquad
\frac12\int_\R(z^2)_x\operatorname{sgn}(z)\,dx
=\frac12\int_\R(z|z|)_x\,dx=0.
\]
Consequently, applying Lemma~\ref{lem:residual-one}, we obtain
\[
\begin{aligned}
\|z(t,\cdot)\|_{L^1(\R)}
&\le\|z(0,\cdot)\|_{L^1(\R)}+C\int_0^t(1+s)^{1-\alpha}\,ds\\
&=\|z(0,\cdot)\|_{L^1(\R)}
+\frac{C}{2-\alpha}\bigl[(1+t)^{2-\alpha}-1\bigr]\\
&\le C(1+t)^{2-\alpha}.
\end{aligned}
\]
The proof is then complete.
\end{proof}

\smallskip
To control the perturbation in the critical weak-Lorentz space, we first estimate its truncations.

\begin{lemma}\label{truncation-lemma}
Assume $z(0, \cdot)\in L^{q_c,\infty}(\R)$. For $k>0$, we set
\[
\Phi_k(r)=(|r|-k)_+,
\quad \text{and} \quad Y_k(t)=\int_\R\Phi_k(z(t,x))\,dx.
\]
Then, for $t \in [0,\infty)$ and $t_0 \in [0,t]$,
\begin{equation}\label{eq:truncation}
Y_k(t)\le Y_k(t_0)
+\int_{t_0}^t\int_{\{|z(s, \cdot)|>k\}}|F(s,x)|\,dx\,ds.
\end{equation}
\end{lemma}

\begin{proof}
We first consider the smooth compactly supported initial perturbations in Lemma~\ref{lem:approximation}. We test \eqref{F-definition} with $\Phi_k'(z)=\operatorname{sgn}(z)\mathbf 1_{\{|z|>k\}}$, using smooth convex approximations and spatial cutoffs to justify the calculation. Applying the convex Kato inequality $\Lambda^\alpha\Phi(z)\le\Phi'(z)\Lambda^\alpha z$ from \cite[Lemma~1]{DI2006} and integrating, we obtain
\[
\int_\R\Lambda^\alpha z\,\Phi_k'(z)\,dx\ge0.
\]
The quadratic flux is a total derivative:
\[
\frac12\int_\R(z^2)_x\Phi_k'(z)\,dx
=\frac12\int_\R\partial_x\Bigl[\operatorname{sgn}(z)(z^2-k^2)_+\Bigr]\,dx=0.
\]
For the mixed flux, we integrate by parts to obtain
\[
\int_\R(wz)_x\Phi_k'(z)\,dx
=\int_\R w_x\bigl(z\Phi_k'(z)-\Phi_k(z)\bigr)\,dx
=k\int_{\{|z|>k\}}w_x\,dx\ge0.
\]
Thus
\[
Y_k'(t)\le\int_{\{|z(t, \cdot)|>k\}}|F(t,x)|\,dx
\]
in the sense of distributions in time, and we obtain \eqref{eq:truncation} by integrating this last inequality.

\smallskip
For general data, we fix $T>0$ and $r\in(q_c,\infty)$ and use Lemma~\ref{lem:approximation}. For $k>0$, we have
\[
|\Phi_k(a)-\Phi_k(b)|\le |a-b|\mathbf1_{\{|a|>k\}\cup\{|b|>k\}}.
\]
By H\"older's inequality and \eqref{eq:approximation-lr-bounds},
\[
\begin{aligned}
&\int_\R\bigl|\Phi_k(u_n(s,x)-w(s,x))-\Phi_k(z(s,x))\bigr|\,dx\\
&\quad\le C_Tk^{1-r}\|u_n(s,\cdot)-u(s,\cdot)\|_{L^r(\R)}\longrightarrow0,
\qquad s\in[0,T].
\end{aligned}
\]
By \eqref{eq:approximation-convergence}, we choose a subsequence such that $u_n(s,x)\to u(s,x)$ for almost every $(s,x)\in(0,T)\times\R$. For every $\varepsilon>0$,
\[
\limsup_{n\to\infty}\mathbf1_{\{|u_n(s,x)-w(s,x)|>k+\varepsilon\}}
\le\mathbf1_{\{|z(s,x)|\ge k+\varepsilon\}}
\le\mathbf1_{\{|z(s,x)|>k\}}
\quad\text{a.e.}
\]
Since $F\in L^1((0,T)\times\R)$, we apply Fatou's lemma to
$|F|\bigl(1-\mathbf1_{\{|u_n-w|>k+\varepsilon\}}\bigr)$ and obtain
\[
\limsup_{n\to\infty}\int_{t_0}^t\!\int_{\{|u_n(s,\cdot)-w(s,\cdot)|>k+\varepsilon\}}|F(s,x)|\,dx\,ds
\le\int_{t_0}^t\!\int_{\{|z(s,\cdot)|>k\}}|F(s,x)|\,dx\,ds.
\]
Passing to the limit in the truncation inequality for $u_n-w$ at level $k+\varepsilon$, we obtain
\[
Y_{k+\varepsilon}(t)\le Y_{k+\varepsilon}(t_0)
+\int_{t_0}^t\!\int_{\{|z(s,\cdot)|>k\}}|F(s,x)|\,dx\,ds.
\]
As $\varepsilon\downarrow0$, we have $\Phi_{k+\varepsilon}(z)\uparrow\Phi_k(z)$ and hence $Y_{k+\varepsilon}(s)\uparrow Y_k(s)$ for $s\in\{t_0,t\}$. Thus we obtain \eqref{eq:truncation} for every $k>0$ and $0\le t_0\le t\le T$.
\end{proof}

Combining \eqref{eq:truncation} with \eqref{eq:residual-critical}, we obtain logarithmic control of the weak critical norm.

\begin{lemma}\label{lem:weak-critical}
If $z(0, \cdot)\in L^{q_c,\infty}(\R)$, then
\begin{equation}\label{eq:weak-critical}
\|z(t, \cdot)\|_{\weakL{q_c}}
\le C\Bigl(\|z(0, \cdot)\|_{\weakL{q_c}}+\log(1+t)\Bigr),\qquad t\ge0.
\end{equation}
Here $C$ depends only on $\alpha$ and the end states $u_{\pm}$.
\end{lemma}

\begin{proof}
We fix $T>0$ and first assume additionally that $z(0,\cdot)\in L^1(\R)$. By Lemma~\ref{lem:perturbation-one} and \eqref{eq:bounded-profile-perturbation},
\[
\sup_{\tau\in[0,T]}\|z(\tau,\cdot)\|_{\weakL{q_c}}\le\sup_{s\in[0,T]}
\|z(s,\cdot)\|_{L^1(\R)}^{1/q_c}
\|z(s,\cdot)\|_{L^\infty(\R)}^{1-1/q_c}<\infty.
\] By \eqref{eq:weak-set-estimate} and \eqref{eq:residual-critical},
\[
\begin{aligned}
\int_{\{|z(s,\cdot)|>k\}}|F(s,x)|\,dx
&\le C\|F(s,\cdot)\|_{\weakL{q_c}}
|\{|z(s,\cdot)|>k\}|^{1-1/q_c}\\
&\le C(1+s)^{-1}\bigl(\sup_{\tau\in[0,T]}\|z(\tau,\cdot)\|_{\weakL{q_c}}^{q_c}k^{-q_c}\bigr)^{1-1/q_c}\\
&=C(1+s)^{-1}\sup_{\tau\in[0,T]}\|z(\tau,\cdot)\|_{\weakL{q_c}}^{q_c-1}k^{1-q_c}.
\end{aligned}
\]
Also, we note that
\[
Y_k(0)=\int_k^\infty|\{|z(0, \cdot)|>\lambda\}|\,d\lambda
\le\frac{\|z(0,\cdot)\|_{\weakL{q_c}}^{q_c}}{q_c-1}k^{1-q_c},
\]
where $Y_k$ is defined in Lemma \ref{truncation-lemma}. 
Since $k|\{|z(t, \cdot)|>2k\}|\le Y_k(t)$, we infer from Lemma~\ref{truncation-lemma} that
\[
k^{q_c}|\{|z(t, \cdot)|>2k\}|
\le C \|z(0,\cdot)\|_{\weakL{q_c}}^{q_c}+C \sup_{\tau\in[0,T]}\|z(\tau,\cdot)\|_{\weakL{q_c}}^{q_c-1}\log(1+T),
\]
for $t \in [0,T]$. Taking the supremum over $k>0$ and $t \in [0,T]$ and applying Young's inequality, we obtain
\[
\sup_{\tau\in[0,T]}\|z(\tau,\cdot)\|_{\weakL{q_c}}^{q_c}\le C \|z(0,\cdot)\|_{\weakL{q_c}}^{q_c}+\frac12\sup_{\tau\in[0,T]}\|z(\tau,\cdot)\|_{\weakL{q_c}}^{q_c}
+C[\log(1+T)]^{q_c}.
\]
Subtracting the term with coefficient $1/2$ and taking the $q_c$-th root, we obtain
\[
\sup_{\tau\in[0,T]}\|z(\tau,\cdot)\|_{\weakL{q_c}}
\le C\Bigl[\|z(0,\cdot)\|_{\weakL{q_c}}+\log(1+T)\Bigr].
\]

\smallskip
For general data, we fix $t>0$ and choose $(u_{n_k})$ realizing the lower limit of the weak norms. By Lemma~\ref{lem:approximation}, we can choose these indices so that
\[
u_{n_k}(t,x)\to u(t,x)\quad\text{for almost every }x\in\R.
\]
Applying Fatou's lemma to the level sets, we obtain
\[
\begin{aligned}
\|z(t,\cdot)\|_{\weakL{q_c}}
&\le\liminf_{n\to\infty}\|u_n(t,\cdot)-w(t,\cdot)\|_{\weakL{q_c}}\\
&\le C\Bigl(\|z(0,\cdot)\|_{\weakL{q_c}}+\log(1+t)\Bigr).
\end{aligned}
\]
\end{proof}

At the critical exponent $q_c$, we use a uniform $L^\infty(\R)$-bound to obtain a strong $L^{q_c}(\R)$ estimate with an additional logarithmic factor.

\begin{lemma}\label{lem:strong-critical}
Assume $z(0, \cdot)\in L^1(\R)\cap L^{q_c,\infty}(\R)$. 
Then
\begin{equation}\label{eq:strong-critical}
\|z(t, \cdot)\|_{L^{q_c}(\R)}\le C\bigl[1+\log(2+t)\bigr]^\alpha,\qquad t\ge0.
\end{equation}
\end{lemma}
\begin{proof}
We note from \eqref{eq:bounded-profile-perturbation} that
and
\begin{equation}\label{eq:uniform-perturbation-bound}
\sup_{t\ge0}\|z(t, \cdot)\|_{L^\infty(\R)}<\infty.
\end{equation}
We choose $M\ge1$ such that $\|z(t, \cdot)\|_{L^\infty(\R)}\le M$ for all $t$. We split the integral below at $(1+t)^{-(\alpha-1)}$. Applying Lemmas~\ref{lem:perturbation-one} and~\ref{lem:weak-critical}, we obtain
\[
\begin{aligned}
\|z(t, \cdot)\|_{L^{q_c}(\R)}^{q_c}
&=q_c\int_0^M\lambda^{q_c-1}|\{x\in\R:|z(t,x)|>\lambda\}|\,d\lambda\\
&\le q_c\|z(t, \cdot)\|_{L^1(\R)}
\int_0^{(1+t)^{-(\alpha-1)}}\lambda^{q_c-2}\,d\lambda+q_c\|z(t, \cdot)\|_{\weakL{q_c}}^{q_c}
\int_{(1+t)^{-(\alpha-1)}}^M\frac{d\lambda}{\lambda}\\
&\le\frac{q_c}{q_c-1}\|z(t,\cdot)\|_{L^1(\R)}(1+t)^{-(2-\alpha)}\\
&\qquad+q_c\|z(t,\cdot)\|_{\weakL{q_c}}^{q_c}
\bigl[\log M+(\alpha-1)\log(1+t)\bigr]\\
&\le C+C\bigl[1+\log(2+t)\bigr]^{q_c+1}.
\end{aligned}
\]
Here $(\alpha-1)(q_c-1)=2-\alpha$, so the first term is bounded uniformly in time. Taking the $q_c$-th root and using $1+1/q_c=\alpha$, we obtain the result.
\end{proof}

\section{Proof of main results}\label{sec:decay}

We first estimate the difference from the viscous profile. Throughout this section, $w(t, \cdot)=V(t+1, \cdot)$ is the shifted profile in Corollary~\ref{cor:shifted-viscous-profile}, and $z=u-w$.

\begin{lemma}\label{lem:supercritical-decay}
Let $p \in (q_c,\infty]$ and let $u$ be the bounded solution defined by Theorem~\ref{thm:existence}. If
\[
u(0, \cdot)-w(0, \cdot)\in L^p(\R)\cap L^{q_c,\infty}(\R),
\]
then
\[
\|z(t, \cdot)\|_{L^p(\R)}
\le C_p(1+t)^{-\frac{\alpha-1-1/p}{\alpha}}
\bigl[1+\log(2+t)\bigr],\qquad t\ge0
\]
where
\[
C_p= C\Bigl(
1+\|z(0, \cdot)\|_{L^p(\R)}+\|z(0, \cdot)\|_{\weakL{q_c}}\Bigr)
\]
with $C = C(\alpha,p,u_-,u_+)>0$.
\end{lemma}

\begin{proof}[Proof of Lemma~\ref{lem:supercritical-decay}] By \eqref{eq:approximation-lr-bounds} and \eqref{eq:bounded-profile-perturbation}, for $T>0$ and $p\in(q_c,\infty]$,
\[
\sup_{\tau\in[0,T]}\|z(\tau,\cdot)\|_{L^p(\R)}<\infty.
\]
By \eqref{eq:shifted-viscous-bounds}, \eqref{eq:fractional-interpolation}, and \eqref{eq:residual-one}, we also have $\sup_{\tau\in[0,T]}\|F(\tau,\cdot)\|_{L^p(\R)}<\infty$. Using the mild form of \eqref{F-definition} and the kernel derivative estimate in \cite[eq.~(2.4)]{DGV2003}, we obtain
\[
\begin{aligned}
\|z(s,\cdot)-e^{-s\Lambda^\alpha}z(0,\cdot)\|_{L^p(\R)}
&\le\frac{C_\alpha}{2}\int_0^s(s-\tau)^{-1/\alpha}
\|[(u+w)z](\tau,\cdot)\|_{L^p(\R)}\,d\tau\\
&\quad+\int_0^s\|F(\tau,\cdot)\|_{L^p(\R)}\,d\tau
\le C_T\bigl(s^{1-1/\alpha}+s\bigr).
\end{aligned}
\]
Since $e^{-s\Lambda^\alpha}$ is a contraction on $L^p(\R)$,
\begin{equation}\label{eq:perturbation-initial-norm}
\limsup_{s\downarrow0}\|z(s,\cdot)\|_{L^p(\R)}
\le\|z(0,\cdot)\|_{L^p(\R)},\qquad p\in(q_c,\infty].
\end{equation}

\smallskip
\noindent{\bf Case 1}: We consider $p \in (q_c,\infty)$. On a positive time interval, we test \eqref{F-definition} with $|z|^{p-2}z$, using spatial cutoffs and then letting their radius tend to infinity. Applying the Stroock--Varopoulos inequality (see \cite[Lemma 2.3]{KMX2008}), we obtain
\begin{equation} \label{diff-ineq}
 \frac1p\frac{d}{dt}\|z(t, \cdot)\|_{L^p(\R)}^p
+c_{\alpha,p}\bigl\|\Lambda^{\alpha/2}|z(t, \cdot)|^{p/2}\bigr\|_{L^2(\R)}^2+\frac{p-1}{p}\int_\R w_x|z|^p\,dx
\le\Bigl|\int_\R F|z|^{p-2}z\,dx\Bigr|.
\end{equation}
To see this, we note that for $p \in (1,2)$, we use the regularized test $z(|z|^2+\delta)^{(p-2)/2}$ and then let $\delta\downarrow0$. In addition, for the flux terms, we compute
\[
\begin{aligned}
\frac12\int_\R(z^2)_x|z|^{p-2}z\,dx
&=\frac1{p+1}\int_\R\partial_x(z|z|^p)\,dx=0,\\
\int_\R(wz)_x|z|^{p-2}z\,dx
&=\int_\R w_x|z|^p\,dx+\frac1p\int_\R w\partial_x|z|^p\,dx=\frac{p-1}{p}\int_\R w_x|z|^p\,dx\ge0.
\end{aligned}
\]

\smallskip
We also note that by integrating \eqref{diff-ineq}  over $[s,T]$ with $s>0$ and applying H\"older's inequality, we obtain
\[
\begin{aligned}
c_{\alpha,p}\int_s^T\bigl\|\Lambda^{\alpha/2}|z(\tau,\cdot)|^{p/2}\bigr\|_{L^2(\R)}^2\,d\tau
&\le\frac1p\|z(s,\cdot)\|_{L^p(\R)}^p\\
&\quad+\int_s^T\|F(\tau,\cdot)\|_{L^p(\R)}
\|z(\tau,\cdot)\|_{L^p(\R)}^{p-1}\,d\tau<\infty.
\end{aligned}
\]
Thus, from now on, Lemma~\ref{interpolation} applies to $z(t,\cdot)$ for almost every $t>0$.

\smallskip
We multiply the differential inequality \eqref{diff-ineq} by $(1+t)^{\frac{p+1}{q_c+1}}$ and discard the nonnegative term involving $w_x$. We then obtain
\begin{align} \notag
&\frac1p\frac{d}{dt}\bigl((1+t)^{\frac{p+1}{q_c+1}}\|z(t, \cdot)\|_{L^p(\R)}^p\bigr)+c_{\alpha,p}(1+t)^{\frac{p+1}{q_c+1}}\bigl\|\Lambda^{\alpha/2}|z(t, \cdot)|^{p/2}\bigr\|_{L^2(\R)}^2\\ \label{diff-eqn-oct-28}
&\le(1+t)^{\frac{p+1}{q_c+1}}\Bigl|\int_\R F|z|^{p-2}z\,dx\Bigr|
+\frac{p+1}{p(q_c+1)}(1+t)^{\frac{p-q_c}{q_c+1}}\|z(t, \cdot)\|_{L^p(\R)}^p.
\end{align}
We apply \eqref{eq:weighted-interpolation} and \eqref{eq:forcing-interpolation} to the two terms on the right hand side of \eqref{diff-eqn-oct-28}, using
$\|F(t, \cdot)\|_{\weakL{q_c}}\le C(1+t)^{-1}$ for the first and
$\lambda=\frac{p+1}{p(q_c+1)(1+t)}$ for the second. The sum of the two terms in the right hand side of \eqref{diff-eqn-oct-28} is at most
\[
2\varepsilon(1+t)^{\frac{p+1}{q_c+1}}\bigl\|\Lambda^{\alpha/2}|z(t, \cdot)|^{p/2}\bigr\|_{L^2(\R)}^2+C_{p,\varepsilon}\bigl(\|z(t, \cdot)\|_{\weakL{q_c}}^{\frac{pq_c-1}{q_c+1}}+\|z(t, \cdot)\|_{\weakL{q_c}}^p\bigr).
\]
We choose $\varepsilon=c_{\alpha,p}/4$, subtract the two dissipation terms, and discard the remaining nonnegative term. Since $\frac{pq_c-1}{q_c+1}\in(0,p)$, we apply \eqref{eq:weak-critical} to \eqref{diff-eqn-oct-28} and infer that
\[
\frac{d}{dt}\bigl((1+t)^{\frac{p+1}{q_c+1}}\|z(t, \cdot)\|_{L^p(\R)}^p\bigr)
\le C_p\bigl[1+\|z(0, \cdot)\|_{\weakL{q_c}}+\log(2+t)\bigr]^p,
\]
where $C_p$ depends only on $\alpha$, $p$, and the end states $u_{\pm}$. Integrating from $s>0$ to $t$ and then letting $s\downarrow0$ in \eqref{eq:perturbation-initial-norm}, we obtain
\[
\begin{aligned}
\|z(t, \cdot)\|_{L^p(\R)}^p
&\le(1+t)^{-\frac{p+1}{q_c+1}}\|z(0, \cdot)\|_{L^p(\R)}^p\\
&\quad+C_p t(1+t)^{-\frac{p+1}{q_c+1}}
\bigl[1+\|z(0, \cdot)\|_{\weakL{q_c}}+\log(2+t)\bigr]^p.
\end{aligned}
\]
We take the $p$-th root and use the identity $\frac{p-q_c}{p(q_c+1)}
=\frac{\alpha-1-1/p}{\alpha}$ to conclude that
\[
\|z(t, \cdot)\|_{L^p(\R)}
\le C_p(1+t)^{-\frac{\alpha-1-1/p}{\alpha}}
\bigl[1+\|z(0, \cdot)\|_{L^p(\R)}
+\|z(0, \cdot)\|_{\weakL{q_c}}+\log(2+t)\bigr].
\]
This proves the stated dependence on the norms of the initial perturbation for finite $p$.

\medskip
\noindent{\bf Case 2}: We consider $p=\infty$. We use a one-sided slope estimate of Oleinik type \cite{Oleinik1957}, which we verify for \eqref{eq:fractional-burgers}. Differentiating the equation, we obtain
\[
(u_x)_t+\Lambda^\alpha u_x+u(u_x)_x+(u_x)^2=0.
\]
We fix $0<s<T$ and choose $0<\varepsilon<(1+\|u_x(s,\cdot)\|_{L^\infty(\R)})^{-1}$. For $t\in[s,T]$, we compute
\[
(\partial_t+\Lambda^\alpha+u\partial_x)
\bigl(u_x-\frac1{t-s+\varepsilon}\bigr)
=-\bigl(u_x+\frac1{t-s+\varepsilon}\bigr)
\bigl(u_x-\frac1{t-s+\varepsilon}\bigr).
\]
We apply Lemma~\ref{lem:linear-comparison} on $[s,T]$ with
\[
h=u_x-\frac1{t-s+\varepsilon},\qquad b=u,\qquad
c=u_x+\frac1{t-s+\varepsilon}.
\]
By Theorem~\ref{thm:existence}, the coefficients are bounded on $[s,T]\times\R$, and $h(s,x)=u_x(s,x)-\varepsilon^{-1}<0$. Hence
\[
u_x(t,x)\le\frac1{t-s+\varepsilon},\qquad t\in[s,T].
\]
Letting $\varepsilon\downarrow0$ and then $s\downarrow0$, and using $w_x\ge0$, we obtain
\begin{equation}\label{eq:one-sided-perturbation}
z_x(t,x)=u_x(t,x)-w_x(t,x)\le\frac1t,\qquad t>0.
\end{equation}

\smallskip
We fix $t>0$ and $x$ with $z(t,x)\ne0$. Integrating \eqref{eq:one-sided-perturbation}, we find
\[
z(t,y) \ge z(t,x)-(x-y)/t>z(t,x)/2
\ \text{if }z(t,x)>0, \ x-tz(t,x)/2<y<x,
\]
and
\[ 
z(t,y) \le z(t,x)+(y-x)/t<z(t,x)/2
\ \text{if }z(t,x)<0,\ x<y<x+t|z(t,x)|/2.
\]
In either case,
\[
\begin{aligned}
\|z(t,\cdot)\|_{\weakL{q_c}}^{q_c}
&\ge\Bigl(\frac{|z(t,x)|}{2}\Bigr)^{q_c}
\bigl|\{y:|z(t,y)|>|z(t,x)|/2\}\bigr|\\
&\ge\frac{t}{2^{q_c+1}}|z(t,x)|^{q_c+1}.
\end{aligned}
\]
Taking the supremum over $x$ and applying \eqref{eq:weak-critical}, we obtain
\begin{equation}\label{eq:positive-time-sup}
\begin{aligned}
\|z(t,\cdot)\|_{L^\infty(\R)}
&\le2t^{-1/(q_c+1)}\|z(t,\cdot)\|_{\weakL{q_c}}^{q_c/(q_c+1)}\\
&\le Ct^{-\frac{\alpha-1}{\alpha}}
\bigl[\|z(0,\cdot)\|_{\weakL{q_c}}+\log(1+t)\bigr]^{1/\alpha},
\qquad t>0.
\end{aligned}
\end{equation}
Here we used $1/(q_c+1)=(\alpha-1)/\alpha$ and $q_c/(q_c+1)=1/\alpha$. In particular, the constant is independent of $\|z(0,\cdot)\|_{L^\infty(\R)}$.

\smallskip
For $t\in[0,1]$, we use \eqref{eq:fractional-interpolation} with $p=\infty$ and \eqref{eq:shifted-viscous-bounds} to obtain $\|F(t,\cdot)\|_{L^\infty(\R)}\le C$. We write \eqref{F-definition} as
\[
z_t+\Lambda^\alpha z+uz_x+w_xz=F.
\]
For $0<s<t\le1$, we set
\[
h(t,x)=\pm z(t,x)-\|z(s,\cdot)\|_{L^\infty(\R)}-C(t-s).
\]
Since $w_x\ge0$ and $|F|\le C$, we have
\[
h_t+\Lambda^\alpha h+uh_x+w_xh
=\pm F-C-w_x\bigl(\|z(s,\cdot)\|_{L^\infty(\R)}+C(t-s)\bigr)\le0,
\qquad h(s,x)\le0.
\]
Applying Lemma~\ref{lem:linear-comparison} on $[s,1]$, we obtain
\[
\|z(t,\cdot)\|_{L^\infty(\R)}
\le\|z(s,\cdot)\|_{L^\infty(\R)}+C(t-s).
\]
Letting $s\downarrow0$ and using \eqref{eq:perturbation-initial-norm}, we conclude
\begin{equation}\label{eq:sup-small-time}
\|z(t,\cdot)\|_{L^\infty(\R)}
\le\|z(0,\cdot)\|_{L^\infty(\R)}+Ct,\qquad t\in[0,1].
\end{equation}
Combining \eqref{eq:positive-time-sup} for $t\ge1$ with \eqref{eq:sup-small-time}, we conclude
\begin{align} \notag
\|z(t,\cdot)\|_{L^\infty(\R)}
&\le C\Bigl(1+\|z(0,\cdot)\|_{L^\infty(\R)}
+\|z(0,\cdot)\|_{\weakL{q_c}}\Bigr)\\  \label{sup-z-est-oct-18}
&\quad \qquad\times(1+t)^{-\frac{\alpha-1}{\alpha}}[1+\log(2+t)],\qquad t\ge0.
\end{align}
The proof is complete.
\end{proof}

\smallskip
Now, we provide the proof of our main theorem.
\begin{proof}[Proof of Theorem~\ref{thm:optimality}] As $w(t, \cdot)=V(t+1, \cdot)$, it follows from Corollary~\ref{cor:shifted-viscous-profile} that $w(0, \cdot) - w^R(0, \cdot) \in L^1\cap L^\infty$. Hence, we infer from the hypotheses that $u_0-w(0, \cdot)\in L^p\cap L^{q_c,\infty}$.
\smallskip
Lemma~\ref{lem:supercritical-decay} now applies to $z=u-w$. Since $w(0, \cdot)-w^R(0, \cdot)$ depends only on the end states $u_{\pm}$, we have
\[
\begin{aligned}
&1+\|z(0, \cdot)\|_{L^p(\R)}+\|z(0, \cdot)\|_{\weakL{q_c}}\\
&\quad \qquad\le C\Bigl(
1+\|u_0-w^R(0, \cdot)\|_{L^p(\R)}
+\|u_0-w^R(0, \cdot)\|_{\weakL{q_c}}\Bigr)
\end{aligned}
\]
with $C=C(\alpha,p,u_-,u_+)>0$. Applying Lemma~\ref{lem:supercritical-decay}, we therefore obtain, for $t\ge0$,
\[
\begin{aligned}
\|z(t,\cdot)\|_{L^p(\R)}
&\le C\Bigl(
1+\|u_0-w^R(0,\cdot)\|_{L^p(\R)}
+\|u_0-w^R(0,\cdot)\|_{\weakL{q_c}}\Bigr)\\
&\quad \qquad\times(1+t)^{-\frac{\alpha-1-1/p}{\alpha}}[1+\log(2+t)].
\end{aligned}
\]
Then, we apply the triangle inequality to obtain
\[
\|u(t, \cdot)-w^R(t, \cdot)\|_{L^p(\R)}
\le \|z(t, \cdot)\|_{L^p(\R)}+\|w(t, \cdot)-w^R(t, \cdot)\|_{L^p(\R)}.
\]
By \eqref{eq:shifted-profile-approximation}, the second term decays faster than the bound for the first because
\[
\frac12\Big(1-\frac1p\Big)
-\frac{\alpha-1-1/p}{\alpha}
=\frac{(2-\alpha)(1+1/p)}{2\alpha}>0.
\]

\smallskip
To remove the dependence on $\|z(0, \cdot)\|_{L^p(\R)}$ for $t\ge1$, we use \eqref{eq:weak-critical} and \eqref{eq:positive-time-sup}. For $p \in (q_c,\infty)$, we compute
\[
\begin{aligned}
\|z(t,\cdot)\|_{L^p(\R)}^p
&=p\int_0^\infty\lambda^{p-1}|\{x\in\R:|z(t,x)|>\lambda\}|\,d\lambda\\
&\le p\|z(t,\cdot)\|_{\weakL{q_c}}^{q_c}
\int_0^{\|z(t,\cdot)\|_{L^\infty(\R)}}\lambda^{p-q_c-1}\,d\lambda\\
&=\frac{p}{p-q_c}\|z(t,\cdot)\|_{\weakL{q_c}}^{q_c}
\|z(t,\cdot)\|_{L^\infty(\R)}^{p-q_c}.
\end{aligned}
\]
Taking the $p$-th root, we obtain
\[
\|z(t, \cdot)\|_{L^p(\R)}
\le\Big(\frac{p}{p-q_c}\Big)^{1/p}
\|z(t, \cdot)\|_{\weakL{q_c}}^{q_c/p}
\|z(t, \cdot)\|_{L^\infty(\R)}^{1-q_c/p}.
\]
Since
\[
\frac{\alpha-1}{\alpha}\Big(1-\frac{q_c}{p}\Big)
=\frac{\alpha-1-1/p}{\alpha},
\]
we obtain, also for $p=\infty$ by the supremum estimate \eqref{sup-z-est-oct-18},
\[
\|z(t, \cdot)\|_{L^p(\R)}
\le C t^{-\frac{\alpha-1-1/p}{\alpha}}
\bigl[1+\|z(0, \cdot)\|_{\weakL{q_c}}+\log(2+t)\bigr],\qquad t\ge1.
\]
We again compare $z$ with $u-w^R$ and absorb the fixed profile contribution into the constant. This proves the final assertion about the dependence on the initial data.
\end{proof}

\smallskip
\begin{proof}[Proof of Remark~\ref{rem:below-critical}] The first assertion of the remark follows from \eqref{eq:perturbation-one}.  In addition, the estimate of $\|z(t,\cdot)\|_{L^{q_c}(\R)}$ follows from \eqref{eq:strong-critical} of Lemma~\ref{lem:strong-critical}. For $p \in (1,q_c)$,  by applying \eqref{eq:weak-critical}, and \eqref{eq:below-interpolation}, we obtain
\[
\|z(t, \cdot)\|_{L^p(\R)}
\le C_p(1+t)^{\frac1p-(\alpha-1)}
\bigl[1+\log(2+t)\bigr]^{\frac{1-1/p}{2-\alpha}}.
\]
This proves assertion estimates in the remark.
\end{proof}

\smallskip
We now prove that the algebraic exponent cannot be improved.

\begin{proof}[Proof of Remark \ref{optimal-remark}]
Applying \eqref{eq:burgers-range-bound} to the step data, we obtain $-1\le u\le0$. We verify strict negativity directly. By reflection, a Galilean change of variables, and uniqueness, we have
\[
u(t,x)=-1-u(t,-t-x),
\]
so $u(t,-t/2)=-1/2$ for every $t>0$. If $u(t_*,x_*)=0$ at a positive time, then $u_t(t_*,x_*)=u_x(t_*,x_*)=0$, since $u\le0$ in space and time. However, from \eqref{fractional-def},
\[
\Lambda^\alpha u(t_*,x_*)
=c_\alpha\int_\R\frac{-u(t_*,y)}{|x_*-y|^{1+\alpha}}\,dy>0,
\]
which contradicts the equation. Hence $u(t,x)<0$ for $t>0$.

\smallskip
For $\lambda\ge1$, the function $u_\lambda(s,x)=\lambda^{\alpha-1}u(\lambda^\alpha s,\lambda x)$ solves the same equation and satisfies
\[
u_\lambda(0,x)=-\lambda^{\alpha-1}\mathbf{1}_{(-\infty,0)}(x)
\le-\mathbf{1}_{(-\infty,0)}(x)=u(0,x).
\]
Applying Corollary~\ref{cor:burgers-comparison} to these ordered initial data, we obtain $u_\lambda(s,x)\le u(s,x)$. Taking $s=1$ and $\lambda=t^{1/\alpha}$, we infer
\[
u(t,t^{1/\alpha}y)
\le t^{-\frac{\alpha-1}{\alpha}}u(1,y),
\qquad t\ge1.
\]
Since $u(1,\cdot)$ is continuous and strictly negative, we may set
\[
c=-\max_{y \in [0,1]}u(1,y)>0.
\]
As $w^R(t,x)=0$ for $x\ge0$, it follows that
\begin{equation}\label{eq:optimality-pointwise-lower}
|u(t,x)-w^R(t,x)|
\ge c\,t^{-\frac{\alpha-1}{\alpha}},
\qquad x \in [0,t^{1/\alpha}].
\end{equation}
For finite $p$, we obtain
\[
\|u(t,\cdot)-w^R(t,\cdot)\|_{L^p(\R)}^p
\ge\int_0^{t^{1/\alpha}}c^p t^{-p(\alpha-1)/\alpha}\,dx=c^p t^{-\frac{p(\alpha-1)-1}{\alpha}}.
\]
Taking the $p$-th root proves \eqref{eq:optimality-lower}. For $p=\infty$, we take the essential supremum in \eqref{eq:optimality-pointwise-lower} over $(0,t^{1/\alpha})$ to obtain
\[
\|u(t,\cdot)-w^R(t,\cdot)\|_{L^\infty(\R)}
\ge\operatorname*{ess\,sup}_{x\in(0,t^{1/\alpha})}|u(t,x)-w^R(t,x)|
\ge c\,t^{-\frac{\alpha-1}{\alpha}}.
\] 
The proof is complete.
\end{proof}

\smallskip
We note that at $p=q_c$ the lower bound estimate in Remark \ref{optimal-remark} is constant, and for $p \in [1,q_c)$ it grows with time.  We combine Remark \ref{optimal-remark} and conclude the section with the following remark on the optimality of the decay estimates in Lemma \ref{lem:supercritical-decay}. 
\begin{remark} The algebraic algebraic exponent in the decay estimates in Lemma \ref{lem:supercritical-decay} is optimal.
\end{remark}
\begin{proof} By \eqref{eq:shifted-profile-approximation}, for $p\in(1,\infty]$,
\[
\|w(t, \cdot)-w^R(t, \cdot)\|_{L^p(\R)}
\le C_p(1+t)^{-\frac12(1-1/p)}.
\]
Since
\[
\frac12\Big(1-\frac1p\Big)
-\frac{\alpha-1-1/p}{\alpha}
=\frac{(2-\alpha)(1+1/p)}{2\alpha}>0,
\]
we infer from the reverse triangle inequality \eqref{eq:optimality-lower} that, for each $p \in (q_c,\infty]$ and all sufficiently large $t$,
\[
\|u(t, \cdot)-w(t, \cdot)\|_{L^p(\R)}
\ge\frac c2\,t^{-\frac{\alpha-1-1/p}{\alpha}}
\]
with some constant $c>0$ independent on $t$. Thus the algebraic exponent in the estimates in Lemma \ref{lem:supercritical-decay} is optimal.
\end{proof}

\section{Applications}\label{sec:applications}

In this section, we give two simple direct applications of the decay estimates in Theorem~\ref{thm:optimality}. The first concerns integrable perturbations in one space dimension. The second concerns planar rarefaction waves in two space dimensions. A comprehensive application to a parabolic-elliptic system will be developed in \cite{PS}.

\smallskip
For integrable perturbations of the Riemann step, we can separate the contribution of the initial perturbation from the error generated by fractional diffusion. In this flavor, the following theorem strengthens \cite[Theorem 1.1]{KMX2008}.
\begin{theorem}\label{thm:application-integrable}
Let $\alpha\in(1,2)$ and $u_-<u_+$. Suppose that $u_0\in L^\infty(\R)$ and $u_0-w^R(0,\cdot)\in L^1(\R)$, 
where $w^R$ is given by \eqref{eq:rarefaction}. Then the solution $u$ of \eqref{eq:fractional-burgers} satisfies, for every $p\in(q_c,\infty]$ and $t\ge1$,
\begin{equation}\label{eq:application-integrable}
\begin{aligned}
\|u(t,\cdot)-w^R(t,\cdot)\|_{L^p(\R)}
&\le C t^{-\frac1\alpha(1-1/p)}
\|u_0-w^R(0,\cdot)\|_{L^1(\R)}\\
&\quad+ \bar{C} t^{-\frac{\alpha-1-1/p}{\alpha}}
[1+\log(2+t)],
\end{aligned}
\end{equation}
where $C = C(\alpha,p)>0$ and $\bar{C} = C (\alpha,p,u_-,u_+)>0$  are independent of $u_0$ and $t$.
\end{theorem}

\begin{proof}
We let $v$ solve \eqref{eq:fractional-burgers} with $v(0,\cdot)=w^R(0,\cdot)$. Since this initial datum is nondecreasing, we apply \cite[Lemma 3.1]{KMX2008}, to obtain
\[
\|u(t,\cdot)-v(t,\cdot)\|_{L^p(\R)}
\le C_{\alpha,p}t^{-\frac1\alpha(1-1/p)}
\|u_0-w^R(0,\cdot)\|_{L^1(\R)}.
\]
Applying Theorem~\ref{thm:optimality} to $v$, whose initial perturbation from the Riemann step vanishes, and then using the triangle inequality, we obtain \eqref{eq:application-integrable}.
\end{proof}

\smallskip
Since
\[
\frac{1-1/p}{\alpha}-\frac{\alpha-1-1/p}{\alpha}
=\frac{2-\alpha}{\alpha}>0,
\]
the contribution of the integrable perturbation in \eqref{eq:application-integrable} decays faster than the error associated with the exact Riemann data. Thus we improve both the algebraic decay rate and the range of convergence in \cite[Theorem 1.1]{KMX2008}.

\smallskip
We next consider the two-dimensional equation
\begin{equation}\label{eq:application-planar-equation}
U_t+(-\Delta)^{\alpha/2}U+b\cdot\nabla(U^2)=0,
\qquad (x,y)\in\R^2,\quad t>0,
\end{equation}
where $b=(b_1,b_2)$ and $b_1>0$. Following \cite{KPX2015}, we use the normalized end states $0$ and $1$ and define
\begin{equation}\label{eq:application-planar-wave}
W^R(t,x,y)=
\begin{cases}
0,&x\le0,\\
x/(2b_1t),&0<x<2b_1t,\\
1,&x\ge2b_1t,
\end{cases}
\quad t>0,\qquad
W^R(0,x,y)=\mathbf{1}_{(0,\infty)}(x).
\end{equation}
In \cite[Theorem~2.1, p.~2831]{KPX2015}, the supremum-norm error is bounded by $Ct^{-(\alpha-1)/2}[\log(2+t)]^{(3-\alpha)/2}$. We improve its algebraic exponent under the same initial-data assumptions for $\alpha\in(1,2)$.

\begin{theorem}\label{thm:application-planar}
Let $\alpha\in(1,2)$ and $b=(b_1,b_2)\in\R^2$ with $b_1>0$. Suppose that $U_0\in L^\infty(\R^2)$ and
\[
U_0-W^R(0,\cdot)\in L^1(\R^2).
\]
Then the unique bounded solution of \eqref{eq:application-planar-equation} with initial datum $U_0$, smooth for positive time, satisfies, for $t\ge1$,
\begin{equation}\label{eq:application-planar-decay}
\begin{aligned}
\|U(t,\cdot)-W^R(t,\cdot)\|_{L^\infty(\R^2)}
&\le C t^{-2/\alpha}
\|U_0-W^R(0,\cdot)\|_{L^1(\R^2)}\\
&\quad+ Ct^{-\frac{\alpha-1}{\alpha}}
[1+\log(2+t)],
\end{aligned}
\end{equation}
where constant $C= C(\alpha,b)>$ is independent of $U_0$ and $t$.
\end{theorem}

\begin{proof}
We obtain existence, uniqueness, and positive-time regularity from \cite[Theorem~3.1, p.~2833]{KPX2015} with regularization parameter $\varepsilon=0$, which is allowed for $\alpha\in(1,2]$. We let $v$ be the bounded solution of
\begin{equation}\label{eq:application-planar-profile}
\begin{cases}
v_t+\Lambda^\alpha v+b_1(v^2)_x=0,&(t,x)\in(0,\infty)\times\R,\\
v(0,x)=\mathbf{1}_{(0,\infty)}(x),&x\in\R.
\end{cases}
\end{equation}
and set $\widetilde U(t,x,y)=v(t,x)$. By \cite[Lemma~5.1 and Corollary~5.2, p.~2837]{KPX2015}, the restriction of $(-\Delta)^{\alpha/2}$ to functions independent of $y$ agrees with $\Lambda^\alpha$. Hence $\widetilde U$ solves \eqref{eq:application-planar-equation}. Moreover, $2b_1v$ solves \eqref{eq:fractional-burgers} with nondecreasing initial data. By Corollary~\ref{cor:burgers-comparison}, we have $v_x\ge0$ for positive time, so
\[
b\cdot\nabla\widetilde U=b_1v_x\ge0.
\]
Moreover,
\[
\|U_0-\widetilde U(0,\cdot)\|_{L^1(\R^2)}
=\|U_0-W^R(0,\cdot)\|_{L^1(\R^2)}<\infty.
\]
Applying \cite[Theorem~4.1, eq.~(4.1), p.~2834]{KPX2015} with $p=\infty$ and $\varepsilon=0$, we obtain
\begin{equation}\label{eq:application-planar-stability}
\|U(t,\cdot)-\widetilde U(t,\cdot)\|_{L^\infty(\R^2)}
\le C t^{-2/\alpha}
\|U_0-W^R(0,\cdot)\|_{L^1(\R^2)},
\end{equation}
where $C = C (\alpha,b)>0$.
\smallskip
From \eqref{eq:application-planar-profile}, we obtain
\[
(2b_1v)_t+\Lambda^\alpha(2b_1v)+(2b_1v)(2b_1v)_x
=2b_1\bigl[v_t+\Lambda^\alpha v+b_1(v^2)_x\bigr]=0.
\]
Thus $2b_1v$ solves \eqref{eq:fractional-burgers} with end states $0$ and $2b_1$. Its inviscid rarefaction wave is $2b_1W^R$, viewed as a function of $(t,x)$. Since $2b_1v(0,\cdot)=2b_1W^R(0,\cdot)$, we apply Theorem~\ref{thm:optimality} with $p=\infty$ and divide by $2b_1$ to obtain
\begin{equation}\label{eq:application-planar-profile-decay}
\|\widetilde U(t,\cdot)-W^R(t,\cdot)\|_{L^\infty(\R^2)}
\le C t^{-\frac{\alpha-1}{\alpha}}[1+\log(2+t)],
\qquad t\ge1,
\end{equation}
where $C = C(\alpha, b_1)>0$. Combining \eqref{eq:application-planar-stability} and \eqref{eq:application-planar-profile-decay} by the triangle inequality, we obtain \eqref{eq:application-planar-decay}.
\end{proof}

\smallskip
In the proof of \cite[Theorem~2.1, pp.~2837--2838]{KPX2015}, the authors use \cite[Theorem~1.1]{KMX2008} to estimate the one-dimensional profile error. We replace that step by \eqref{eq:application-planar-profile-decay}. The improvement in the algebraic exponent is
\[
\frac{\alpha-1}{\alpha}-\frac{\alpha-1}{2}
=\frac{(\alpha-1)(2-\alpha)}{2\alpha}>0.
\]
Although the logarithmic power in \eqref{eq:application-planar-decay} is larger, the resulting bound decays strictly faster as $t\to\infty$. We state the conclusion in the supremum norm because a nonzero planar profile error, being independent of $y$, does not belong to $L^p(\R^2)$ for finite $p$.

\end{document}